\documentclass[sn-mathphys]{sn-jnl}% Math and Physical Sciences Reference Style
\usepackage{mathrsfs,yhmath}
\jyear{2023}%

\usepackage{multirow} % Required for multirows
\usepackage{tabularx}
\usepackage{adjustbox}
\usepackage{url}

\usepackage{natbib}
\usepackage{hyperref}
\theoremstyle{thmstyleone}%
\newtheorem{theorem}{Theorem}[section]

\newtheorem{proposition}[theorem]{Proposition}%

\theoremstyle{thmstyletwo}%
\newtheorem{example}{Example}%
\newtheorem{remark}{Remark}%

\theoremstyle{thmstylethree}%
\newtheorem{definition}{Definition}%

\begin{document}

\title[Stable Similarity Comparison]{Stable Similarity Comparison of Persistent Homology Groups}


\author[1]{\fnm{Jiaxing} \sur{He}}\email{547337872@qq.com}

\author*[1]{\fnm{Bingzhe} \sur{Hou}}\email{houbz@jlu.edu.cn}

\author[1,2]{\fnm{Tieru} \sur{Wu}}\email{wutr@jlu.edu.cn}

\author[1]{\fnm{Yang} \sur{Cao}}\email{caoyang@jlu.edu.cn}

\affil*[1]{\orgdiv{School of Mathematics}, \orgname{Jilin University}, \orgaddress{\street{Qianjin Street}, \city{Changchun}, \postcode{130012}, \state{Jilin}, \country{P. R. China}}}

\affil[2]{\orgdiv{School of Artificial Intelligence}, \orgname{Jilin University}, \orgaddress{\street{Qianjin Street}, \city{Changchun}, \postcode{130012}, \state{Jilin}, \country{P. R. China}}}


\abstract{Classification in the sense of similarity is an important issue. In this paper, we study similarity classification in Topological Data Analysis. We define a pseudometric $d_{S}^{(p)}$ to measure the distance between barcodes generated by persistent homology groups of topological spaces, and we provide that our pseudometric $d_{S}^{(2)}$ is a similarity invariant. Thereby, we establish a connection between Operator Theory and Topological Data Analysis. We give the calculation formula of the pseudometric $d_{S}^{(2)}$ $(d_{S}^{(1)})$ by arranging all eigenvalues of matrices determined by barcodes in descending order to get the infimum over all matchings. Since conformal linear transformation is one representative type of similarity transformations, we construct comparative experiments on both synthetic datasets and waves from an online platform to demonstrate that our pseudometric $d_{S}^{(2)}$ $(d_{S}^{(1)})$ is stable under conformal linear transformations,  whereas the bottleneck and Wasserstein distances are not. In particular, our pseudometric on waves is only related to the waveform but is independent on the frequency and amplitude. Furthermore, the computation time for $d_{S}^{(2)}$ $(d_{S}^{(1)})$ is significantly less than the computation time for bottleneck distance and is comparable to the computation time for accelerated Wasserstein distance between barcodes.}




\maketitle

\section{Introduction}
Similarity classification is important in both theoretical research and real world. Topological Data Analysis (TDA) is a ubiquitous tool for classification (See the surveys \cite{Car-book-2009, Zo-book-2005}). Some pseudometrics  have been induced to measure distance between barcodes or persistence modules (see the surveys \cite{Qudot-2017, PRSZ-2020}). However, these pseudometrics can only measure the distances between two barcodes or persistence modules in the sense of congruence. In this paper, we introduce a pseudometric $d_{S}^{(p)}$ to measure the distance between two barcodes of persistence homology groups in the sense of similarity. 

In reality, we often focus on classifying similar objects. For instance, we can easily distinguish the timbre of an instrument, regardless of the pitch and volume of the instrument being played. Meanwhile, even when playing the same piece, we can easily distinguish the timbres of different instruments. The reason for this is that every instrument has its "essential waveform", which means timbres should be determined by the essential waveforms but independent on the frequency and amplitude. We can regard the waveforms of the same instruments as a equivalence class under similarity equivalence, and the essential waveform is the representative element of the equivalence class. More precisely, the essential waveform of an instrument is a similarity invariant.

To present the conclusion more intuitively, we choose two specific instruments, a piano and a tuning fork, and study the essential waveforms in their slices of sound waves.

Firstly, we take the slices from the sound wave of the same instrument. In Figure \ref{fig:piano}, the sound wave \cite{piano} produced by a piano shows a noticeable decrease in amplitude over time, along with changes in frequency.  By selecting slices of the waveform at different time intervals with different length, as shown in Figure \ref{fig:piano_slice}, traditional pseudometrics (such as bottleneck distance and Wasserstein distance) would treat the slices as different instruments which is unreasonable. Our pseudometric, however, can treat these slices as the same instrument.

\begin{figure}[h]
	\centering
	\includegraphics[width=0.95\textwidth]{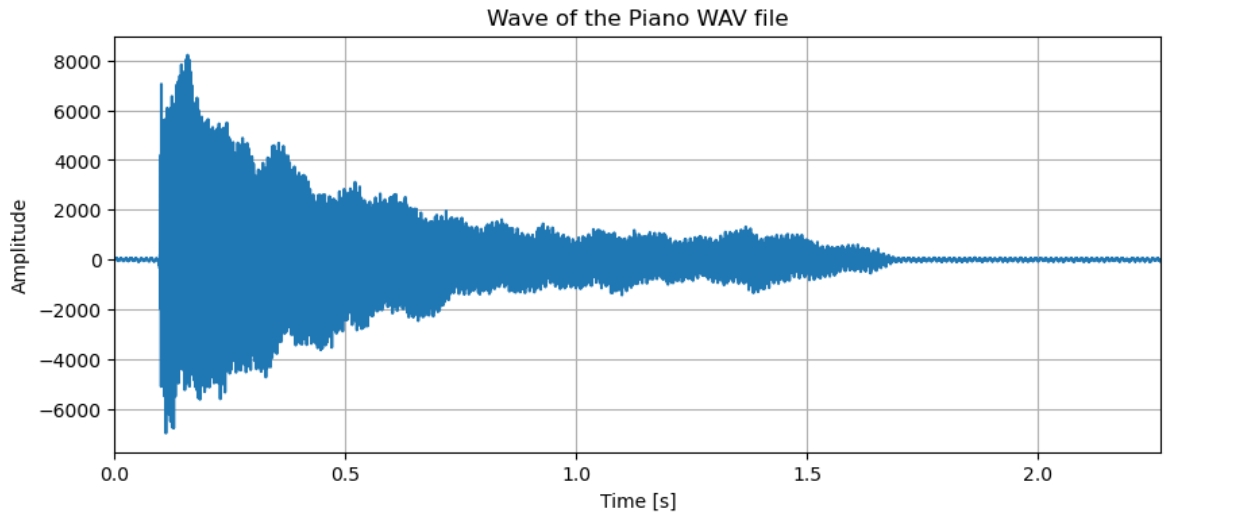}
	\caption{A wave from piano.}
	\label{fig:piano}
\end{figure}

\begin{figure}[h]
	\centering
	\includegraphics[width=0.95\textwidth]{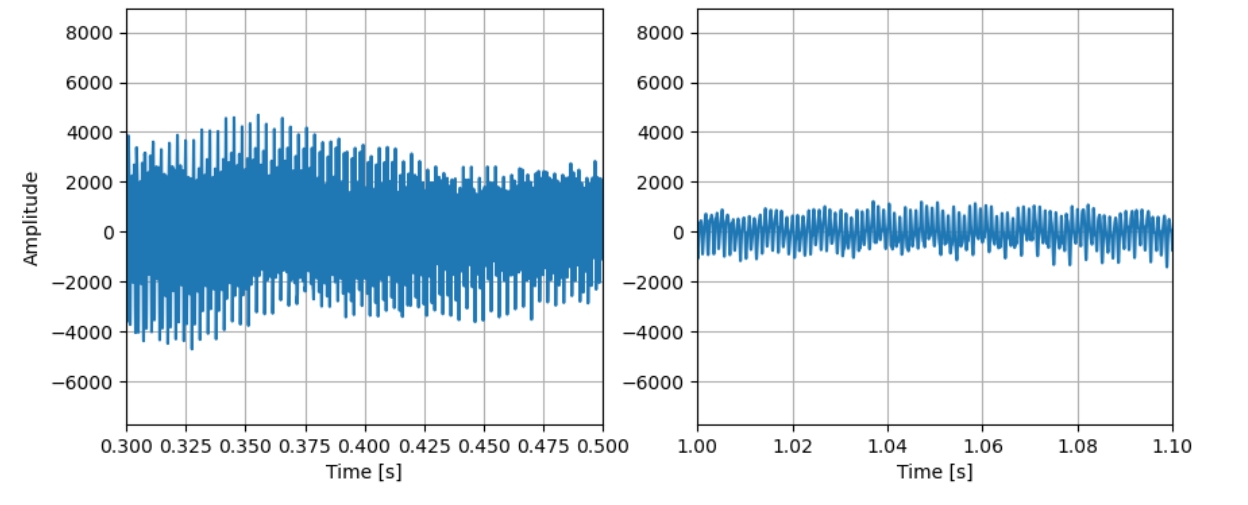}
	\caption{Two segments of different durations extracted from the wave in Figure \ref{fig:piano}.
	}
	\label{fig:piano_slice}
\end{figure}

Secondly, we take the slices of different instruments. The timbres of a tuning fork are different from those of a piano, as the image of the waveform of tuning fork \cite{tuning} shown in Figure \ref{fig:tuning}. Our pseudometric can also distinguish the timbres of the piano and the tuning fork. This is because our pseudometric is only related to the essential waveforms but  independent on the frequencies and amplitudes. More precisely, our pseudometric is stable under similarity.

\begin{figure}[h]
	\centering
	\includegraphics[width=0.95\textwidth]{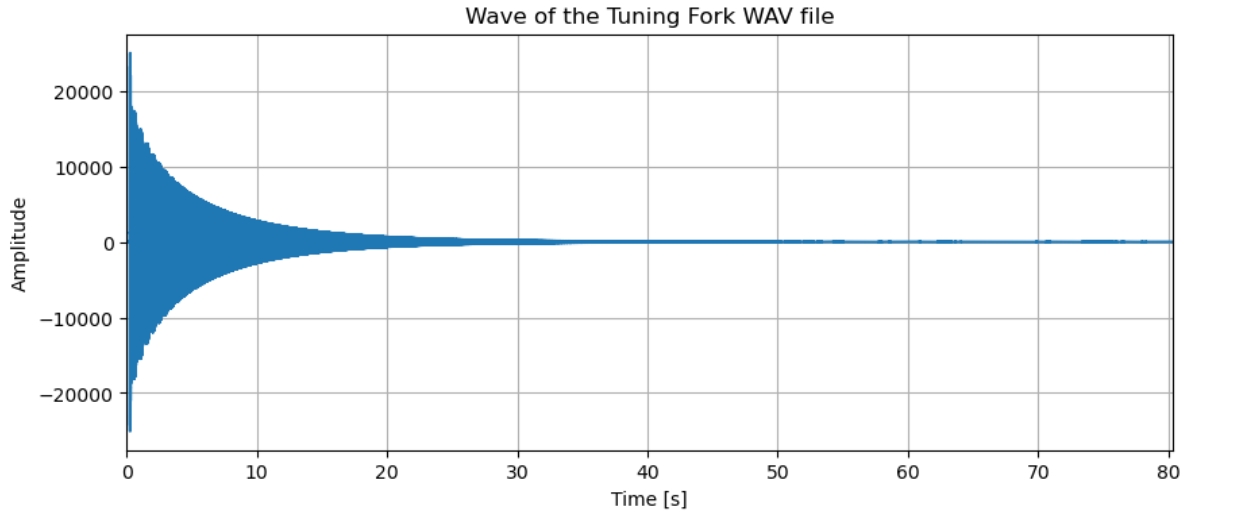}
	\caption{A wave from tuning fork.}
	\label{fig:tuning}
\end{figure}

Similarity classification is an improtant issue in  Topological Data Analysis (TDA). Now, let us overview TDA. Topological Data Analysis (TDA), as an emerging field, has made significant contributions to both theory and applications. Persistent homology as a useful tool in TDA provides insights into the topological features that persist across different scales. A version of homology theory was invented by H. Edelsbrunner, D. Letscher and A. Zomorodian \cite{Edel-Lets-Zo-2000} at the beginning of the millennium.  G. Carlsson and A. Zomorodian  \cite{Zo-Gu-2005} first introduced persistence module for dealing with persistence homology in 2005, and they showed that no similar complete discrete invariant exists for multidimensional persistence in 2007 \cite{Car-2009}. Since TDA encodes topological structures hidden in the data, it can lead to a rich of high efficient algorithms \cite{David-2006} and software \cite{RIVET-2020, Dion-2012}, and a board range of applications including shape description \cite{Co-Zo-Car-2004}, differentiable
topology layers \cite{Gab-Nel-Dwa-2020} as well as holes detecting \cite{De-Gh-Mu-2005}. For recent surveys, we refer to \cite{Edel-Moro-2017, Car-2020, Pun-Lee-Xia-2022}.  

Persistent homology is generated by a filtration, and the persistent homology groups produced by different filtrations show various topological features. In the 1-parameter setting, there are three well-known filtrations on point clouds, the Vietoris-Rips filtration, the \v{C}ech filtration and the alpha filtration. Filtrations on point clouds can be used for classifying dynamical systems \cite{Adams-Emer-Kir-2017}, biomolecule \cite{Meng-Anand-Lu-2020} and proteins \cite{Kov-Bub-Nik-2016}. Filtrations on sublevel sets have been successful applied in the fields of graphics \cite{Pou-Skr-Ovs-2018} and medical imaging \cite{Ahm-Nuw-Tor-2023, Oya-Hir-Oba-2019}.

Barcodes and persistence modules are representations of persistent homology groups. There has been a lot of theoretical research, such like  extension of persistence module concept, spectral sequences and interval decompositions of persistence modules. To see more details, we refer to \cite{Qudot-2017}.

Pseudometrics are usually used to classify barcodes and persistence modules, many scholars have proposed various pseudometrics in recent years. Natural pseudo-distance, bottleneck distance $d_{b}$ and interleaving distance $d_{I}$ are three important pseudometrics, their relations have been considered and Isometry Theorem have been provided which can be seen in \cite{Co-Edel-Har-2007, d-Fro-Lan-2010, Bau-Les-2014, Les-2015}.  P. Frosini introduced a pseudometric $d_{T}$ as a stable method to compare persistent homology groups with torsion in \cite{Fro-2013}. G. Carlsson and B. Filippenko in \cite{Carl-Fil-2020} investigate the persistent homology $PH_{*}(X\times Y)$ of the Cartesian product $X\times Y$ equipped with the sum metric $d_{X}+d_{Y}$ when given finite metric spaces $(X,d_{X})$ and $(Y,d_{Y})$. In \cite{Fro-Lan-Me-2019}, P. Frosini, C. Landi and F. M{\'e}moli introduced the persistent homotopy type distance $d_{HT}$ 
to compare two real valued functions defined on possibly different homotopy equivalent topological spaces.

In the real world, similar objects always need to be classified into one class.  More precisely, the objects under similarity transformations should belong to the same class. However, the pseudometrics such as $d_{b}$, $d_{I}$ and $d_{T}$ we mentioned  in previous cannot classify objects in the sense of similarity. In this paper, we want to introduce a pseudometric which is different from the existing pseudometrics, to classify the objects effectively in the sense of similarity. That would be a more practical and effective approach.

\subsection{Contributions}
We devote a major portion of this paper to the following contributions:

$\bullet$ We introduce a similarity pseudometric $d_{S}^{(p)}$ that can compare barcodes generated by persistent homology groups on topological spaces. We provide that our pseudometric $d_{S}^{(2)}$ is a similarity invariant.

$\bullet$  We give a calculation formula $d_{S}^{(2)}$ $(d_{S}^{(1)})$. The pseudometric $d_{S}^{(2)}$ $(d_{S}^{(1)})$ is the infimum distance of the eigenvalues of the matchings determined by barcodes over all matchings. In fact, the matching to get the distance is arranging all eigenvalues in descending order of each matrices. 

$\bullet$ To verify the similar objects are close under our pseudometric $d_{S}^{(2)}$ $(d_{S}^{(1)})$, we conduct comparable experiments on  synthetic datasets, and then the experimental results demonstrate that our pseudometric is stable under conformal linear transformations which is one representative type of similarity transformations. 

$\bullet$ We also conduct comparable experiments on  waves of a tuning fork and a piano, and  then the experimental results demonstrate that our pseudometric $d_{S}^{(2)}$ $(d_{S}^{(1)})$ is only related to the essential waveform and is independent on frequency and amplitude.

$\bullet$ Our  computation time  of $d_{S}^{(2)}$ $(d_{S}^{(1)})$ is significantly less than that of bottleneck distance and is comparable to that of the accelerated Wasserstein distance.

Our code for the experiments is publicly available at \url{https://github.com/HeJiaxing-hjx/Stable_Similarity_Comparison_of_Persistent_Homology_Groups}.

\section{Preliminaries}

In this section, we review some background knowledge and introduce some preliminary definitions. 

Let us introduce some definitions of filtration at first. A sequence of nested topological spaces, 
\begin{align}\label{nest}
	X^{1}\subseteq X^{2}\subseteq\cdots\subseteq X^{n}\subseteq\cdots,
\end{align}
is called a filtration.

Let $\varphi$ be a real-valued continuous function on a topological space, $X^{t}=\varphi^{-1}((-\infty,t))$ be a sublevel set, where $t\in\mathbb{R}$, the sequence $\{X^{t}\}_{t\in \mathbb{R}}$ is called a $sublevel$ $set$ $filtration$.

Let $R^{i}(X,d_{X})$  be a simplicial complex on the vertex set $X$ that satisfies the following condition, 
$$
\sigma \in R^{i}(X,d_{X})\Leftrightarrow d_{X}(x,y)\leq \text{i, for all } x,y\in \sigma, i\in \mathbb{R}.
$$
Notice that $R^{i}(X,d_{X})$ is empty for $i<0$ and only consists of the vertex set $X$ for $i=0$. There is a natural inclusion $R^{i}(X,d_{X})\subseteq R^{j}(X,d_{X})$ whenever $i\leqslant j$. Thus, the simplical complexes $R^{i}(X,d_{X})$ together with these inclusion maps define a simplicial filtration $\mathcal{R}(X,d_{X})$ with vertex set $X$, called the $(Vietoris$-$)Rips$ $complex$ of $(X,d_{X})$.

Given a filtration, we can compute its persistent homology group. Now, we will introduce the definition of persistent homology. 

Let $\{X^{l}\}$ be a  sequence of nested topological spaces, where $X^{l}$ is the $l\ th$ topological space. Elements of $C_{n}(X^{l})$ are called $n$-chains. We perform a chain complex $C_{*}(X^{l})$ as follows,
$$
\cdots\rightarrow C_{p+1}(X^{l})\stackrel{\partial_{p+1}^{l}}\rightarrow C_{p}(X^{l})\stackrel{\partial_{p}^{l}}\rightarrow\cdots,
$$
where each $\partial_{p}^{l}$ is a boundary map. To see more details about homology theory, we refer to \cite{Hatcher}.

\begin{definition}
	Let $\{X^{l}\}$ be a filtration. The $k$-th persistent $p$-th homology group of $\{X^{l}\}$ is 
	$$
	H^{l,k}_{p} = Z_{p}^{l}/(B_{p}^{l+k}\cap Z_{p}^{l}),
	$$
	where $Z_{p}^{l}=Z_{p}(X^{l})=\text{ker}\partial_{p}^{l}$, $B_{p}^{l}=B_{p}(X^{l})=\text{im}\partial_{p+1}^{l}$. The $k$-th persistent $p$-th Betti number $\beta^{l,k}_{p}$ is the rank of $H^{l,k}_{p}$. 
\end{definition}
For more details about persistent homology, we refer to \cite{Qudot-2017}.

Given a filtration, persistent homology provides a small description in terms of a multiset of intervals which are called barcodes.  Barcodes could mark the parameter values for births and deaths of homological features. In \cite{Car-2009}, a $multiset$ is a set in which an element may appear multiple times, such as $\{a, a, b, c\}$. 

In \cite{PRSZ-2020}, a $barcode$ $\mathcal{B}$ $of$ $finite$ $type$  is a finite multiset of intervals. i.e., it is a finite collection $\{(I_{i},m_{i})\}$ of intervals $I_{i}$ with given multiplicities $m_{i}\in \mathbb{N}$. The intervals $I_{i}$ are all either finite of the form $(a,b]$ or infinite of the form $(a,+\infty)$. Based on the concept of finite type, we provide an equivalent form of barcode and then present a stronger version named the barcode of finite bounded type.

\begin{definition} 
	A $barcode$ $\mathcal{B}$ $of$ $finite$ $bounded$ $type$  is a finite multiset of intervals. i.e., it is a finite collection of intervals $\{I_{i}\}$ and an interval could appear multiple times. The intervals $I_{i}$ are all  finite bounded of the form $(a,b]$. Each element $I_{i}$ in a barcode will be sometimes is called a $bar$.
\end{definition}

Given a sequence like (\ref{nest}), we want not only to compute the topological structure of each space $X^{i}$ separately, but also to understand how topological features persist across the family. The right tool to do this is homology over field, which turns (\ref{nest}) into a sequence of vector spaces (the homology groups $H_{*}(X^{i})$) connected by linear maps (induced by the inclusions $X^{i}\hookrightarrow X^{i+1}$),
\begin{align}\label{homology}
	H_{*}(X^{1})\rightarrow H_{*}(X^{2})\rightarrow \cdots \rightarrow H_{*}(X^{n}).
\end{align}
Such a sequence is called a $persistence$ $module$. Thus we move from the topological category to the algebraic category, where our initial problem becomes to find bases for the vector spaces $H_{*}(X^{i})$ that are 'comparable' with the maps in (\ref{homology}). Roughly speaking, being compatible means that for any indices $i$, $j$ with $1\leq i \leq j\leq n$, the composition
$$
H_{*}(X^{i})\rightarrow H_{*}(X^{i+1})\rightarrow \cdots \rightarrow H_{*}(X^{j-1}) \rightarrow H_{*}(X^{j})
$$
has a (rectangular) diagonal matrix in the bases of $H_{*}(X^{i})$ and $H_{*}(X^{j})$. Then, every basis element can be tracked across the sequence (\ref{homology}), its $birth$ $time$ $b$ and $dead$ $time$ $d$ defined respectively as the first and last indices at which it is a part of the current basis. At the topological level, this basis element corresponds to some features (connected components, hole, void, etc.) appearing in $X^{b}$ and disappearing in  $X^{d+1}$. Its lifespan is encoded as an interval $(b,d]$ in the persistence barcode.

In the theory of persistent homology, the most widely used distance on barcodes is the bottleneck distance $d_{b}$. Another important pseudometric is the Wasserstein distance $d_{\mathcal{W}}$. To the best of our knowledge, this version of the Wasserstein
distance first appeared in the TDA literature for arbitrary $p\in [1,\infty]$ in \cite{Ro-Tur-2017}. The following definitions of Wasserstein distance and bottleneck distance on barcodes are given in \cite{Bje-Les-2021}.

A $matching$ between sets $\mathcal{S}$ and $ \mathcal{T}$ is a bijection $\sigma:\mathcal{S}'\rightarrow \mathcal{T}'$, where $\mathcal{S}'\subset \mathcal{S}$, and $\mathcal{T}'\subset \mathcal{T}$. Formally, we may regard $\sigma$ as a subset of $\mathcal{S}\times \mathcal{T}$, where
$(s,t)\in \sigma$ if and only if $\sigma(s)=t$. In fact, $\sigma$ is a matching between barcodes. The definition of a matching extends readily to multisets.

Write each interval in a barcode of finite bounded type
in the form $(a_{1}, a_{2}]$, where $a_{1}< a_{2}\in\mathbb{R}$. We  identify this interval
with the point $a = (a_{1}, a_{2}) \in \mathbb{R}\times \mathbb{R}$, and denote $$m(a) = (\frac{a_{1}+a_{2}}{2},\frac{a_{1}+a_{2}}{2})\in \mathbb{R}^{2}.$$
\begin{definition}\cite{Bje-Les-2021}
	Let $\mathcal{B}$ and $\mathcal{C}$ be  barcodes of finite type. Given $t\in [1,\infty)$ and a matching $\sigma :\mathcal{B}\rightarrow \mathcal{C}$, we define
	\begin{align*}
		cost(\sigma, p) = \sqrt[p]{\sum\limits_{(s,t)\in \sigma} \Vert s-t \Vert_{p}^{p} +\sum\limits_{a\in \mathcal{B}\sqcup \mathcal{C} \text{ unmatched~by} ~\sigma}\Vert a-m(a) \Vert_{p}^{p}}.
	\end{align*}
	Similarly,
	\begin{align*}
		cost(\sigma, \infty) = \text{max}\left\{\underset{(s,t)\in \sigma}{\text{max}}\Vert s-t \Vert_{\infty},\underset{a\in \mathcal{B}\sqcup \mathcal{C} \text{ unmatched~by} ~\sigma}{\text{max}}\Vert a-m(a) \Vert_{\infty}\right\}.
	\end{align*}
	For $p\in [1,\infty]$, the $p$-$Wasserstein$ $distance$ $between$ $\mathcal{B}$ $and$ $\mathcal{C}$, denoted $d_{\mathcal{W}}^{(p)}(\mathcal{B},\mathcal{C})$, is defined by 
	\begin{align*}
		d_{\mathcal{W}}^{(p)}(\mathcal{B},\mathcal{C})=\underset{\sigma}{\text{min}}~cost(\sigma, p).
	\end{align*}
	where $\sigma$ is taken all over the matchings between the two barcodes $\mathcal{B}$ and $\mathcal{C}$. It is easily checked that,
	\begin{align*}
		d_{\mathcal{W}}^{\infty}(\mathcal{B},\mathcal{C}) = \lim_{p\rightarrow\infty}d_{\mathcal{W}}^{(p)}(\mathcal{B},\mathcal{C}).
	\end{align*}
\end{definition}
The $bottleneck$ $distance$ between two multisets $\mathcal{B}$, $\mathcal{C}$ is the smallest possible bottleneck cost achieved by partial matchings between them,
\begin{align*}
	d_{b}(\mathcal{B},\mathcal{C}) 
	= & d_{\mathcal{W}}^{\infty}(\mathcal{B},\mathcal{C}) \\
	= & \underset{\sigma}{\text{min}}~\text{max}\{\underset{(s,t)\in \sigma}{\text{max}} {\Vert s-t \Vert}_{\infty}, \underset{a\in \mathcal{B}\sqcup \mathcal{C} \text{ unmatched~by} ~\sigma}{\text{max}}\Vert a-m(a) \Vert_{\infty}\}.
\end{align*}
where $\sigma$ is taken all over the matchings between the two barcodes $\mathcal{B}$ and $\mathcal{C}$.

Next, we will introduce the definitions of the functional analysis related to this article.

An operator $T$ on a Hilbert space $H$ to itself satisfying $T^{*}=T$ is said to be $Hermitian$  or $selfadjoint$. $\mathscr{L}(H)$ is the set of all bounded linear operators on $H$.  $T\in \mathscr{B}(H)$ is said to be normal if $T$ and $T^{*}$ are commutative, i.e.,
$$
T^{*}T = TT^{*}.
$$
Hermitian operators and unitary operators on a Hilbert space into itself are special cases of $normal$ $operators$.

Next, we will introduce several relevant definitions of norms.

\begin{definition}
	For a matrix $A\in \mathbb{R}^{m\times n}$, the Frobenius norm is defined by
	$$
	\Vert A\Vert_{F}=\sqrt{\sum\limits_{i=1}^{m}\sum\limits_{j=1}^{n}|a_{ij}|^{2}}=\sqrt{Tr(A^{*}A)},
	$$
	where $A^{*}$ is the conjugate transpose of $A$ and $Tr$ denotes the trace. We call that	an infinite dimensional matrix $A$ is a  Hilbert-Schmidt operator if $Tr(A^{*}A)<\infty$.
\end{definition}

\begin{definition}
	For a linear operator $T:X\rightarrow Y$, where $X$ and $Y$ are normed spaces, the operator norm is defined by
	$$
	\Vert T \Vert = \underset{\Vert x \Vert_{X}=1}{\text{sup}}\Vert Tx \Vert_{Y}.
	$$ 
	
\end{definition}

To consider similarity invariant, we need the following definition of  similarity orbit.

\begin{definition}
	Let $H $ be a separable Hilbert space of finite dimension or infinite dimension.  Denote by $\mathscr{L}(H)$ the set of all bounded linear operators on $H$. For $A,B\in \mathscr{L}(H)$, we call that $A$ is similar to $B$ if there exists an invertible operator $T\in \mathscr{L}(H)$ such that $TA=BT$, denoted by $A\sim B$. Denote by $\mathscr{S}(T)$ the set of all bounded linear operators on $H$ which are similar to $T$. We call $\mathscr{S}(T)$ the similarity orbit of $T$. Let $$d_{F}(\mathscr{S}(A),\mathscr{S}(B))=\underset{A'\in \mathscr{S}(A),B'\in \mathscr{S}(B)}{\text{inf}}\Vert A'-B' \Vert_{F},$$ where $\Vert\cdot\Vert_{F}$ is the Frobenius norm (if $A'-B'$ does not belong to the Hilbert-Schmidt operators class, $\Vert A'-B' \Vert_{F}$ is infinite). We call $d_{F}(\mathscr{S}(A),\mathscr{S}(B))$ the Frobenius distance between the similarity orbits of $A$ and  $B$.

\end{definition}

In addition, the Hoffman-Wielandt Theorem will be useful to prove the stability of our pseudometric.

\begin{theorem}\label{Hoff}
	Let $A$ and $B$ be two $n\times n$ normal matrices. Suppose their eigenvalues are $\lambda_{1}$, $\lambda_{2}$, $ \ldots$, $\lambda_{n}$ and $\mu_{1}$, $\mu_{2}$, $  \ldots$, $\mu_{n}$, respectively. Then there exists a permutation of $\{1,2, \ldots,n\}$ such that,
	\begin{align*}
		\sqrt{\sum\limits_{i=1}^{n}\vert\mu_{\pi(i)}-\lambda_{i}\vert^{2}}\leq \Vert B-A\Vert_{F}.
	\end{align*}
	
\end{theorem}

\section{The Pseudometric $d_{S}^{(p)}$ and similarity invariance}

It is usual to use the pseudo-distance on persistent barcodes to metric the (filtered) topological spaces. In this section, we introduce the pseudometric $d_{S}$ on  finite type barcodes, which is invariant in the sense of similarity.

Let $\mathcal{B}=\{I_{i}\}_{i=1}^{n}$ be a barcode of finite bounded type, where each $I_{i}$ is a bouneded interval $(a_{i},b_{i}]$ and an interval could appear multiple times in $\mathcal{B}$. Then we set $f_{\mathcal{B}}=\{f_{i}\}_{i=1}^{n}$, where each $f_{i}$ is the characteristic function on the interval $I_{i}=(a_{i},b_{i}]$, i.e.,
$$ f_{i}(x)=\left\{
\begin{aligned}
	&1, ~x \in (a_{i},b_{i}],\\
	&0, ~otherwise.
\end{aligned}
\right.
$$
Furthermore, the sequence of functions $f_{\mathcal{B}}$ naturally induces  the Gram matrix $G_{\mathcal{B}}$ defined by
$$
G_{\mathcal{B}} = \begin{pmatrix}
	\langle f_{1}, f_{1} \rangle & \langle f_{1}, f_{2}\rangle & \cdots & \langle f_{1}, f_{n}\rangle \\
	\langle f_{2}, f_{1}\rangle & \langle f_{2}, f_{2}\rangle & \cdots & \langle f_{2}, f_{n}\rangle \\
	\vdots & \vdots & \ddots & \vdots \\
	\langle f_{n}, f_{1}\rangle & \langle f_{n}, f_{2}\rangle & \cdots & \langle f_{n}, f_{n}\rangle
\end{pmatrix},
$$
where $\langle f_{i}, f_{j} \rangle =\int_{\mathbb{R}} \overline{f_{j}(x)}\cdot f_{i}(x)dx= \int_{(a_{i},b_{i}] \bigcap (a_{j},b_{j}]}1 dx$.

Obviously, the Gram matrix $G_{\mathcal{B}}$ is positive semi-definite which means $x^*G_{\mathcal{B}}x\geq 0$ for any $n$-dimensional vector $x$, and is non-negative which means every element in the matrix is non-negative. As well known, the Perron-Frobenius Theorem tells us that the operator norm of $G_{\mathcal{B}}$ is just the maximal (positive) eigenvalue. In order to eliminate the influence of scale, we consider ${G_{\mathcal{B}}}/{\Vert  G_{\mathcal{B}} \Vert}$, the normalization of $G_{\mathcal{B}}$, where $\Vert G_{\mathcal{B}}\Vert$ is the operator norm of $G_{\mathcal{B}}$.

To compare the matrices of different dimensions, we extend the $n$-dimensional Gram matrix ${G_{\mathcal{B}}}/{\Vert  G_{\mathcal{B}} \Vert}$ to an infinite-dimensional matrix $\widetilde{G}_{\mathcal{B}}$ in a natural way. More precisely, let
\[
\widetilde{G}_{\mathcal{B}}=\frac{G_{\mathcal{B}}}{\Vert  G_{\mathcal{B}} \Vert }\oplus \mathbf{0}=\begin{pmatrix}
	{G_{\mathcal{B}}}/{\Vert  G_{\mathcal{B}} \Vert}  & \mathbf{0} \\
	\mathbf{0}   & \mathbf{0}
\end{pmatrix},
\]
where $\mathbf{0}$ is the zero operator on an infinite-dimensional space. Therefore,
$\widetilde{G}_{\mathcal{B}}$ is an infinite-dimensional semi-definite  non-negative matrix with only finite nonzero elements. Let $\Lambda_{\mathcal{B}} = \{\lambda_{i}\}_{i=1}^{\infty}$ be the eigenvalues of $\widetilde{G}_{\mathcal{B}}$. One can see that there are only finite nonzero (positive) eigenvalues.

Let  $\Lambda = \{\lambda_{i}\}_{i=1}^{\infty}$ and $\Gamma = \{\mu_{j}\}_{j=1}^{\infty}$ be two sequences of non-negative numbers with finite nonzero (positive) elements, respectively. A one-one correspondence $\sigma: \mathbb{N}\rightarrow \mathbb{N}$ induces a matching between the two sequences. Then, for any $1\le p \le \infty$, the infimum of the distance over all matchings,
\[
d_{S}^{(p)}(\Lambda, \Gamma)=\inf\limits_{\sigma}\sqrt[p]{\sum\limits_{i=1}^{\infty} \vert\lambda_{i}-\mu_{\sigma(i)}\vert^{p}}
\]
is a distance between two sequences. Furthermore, for any $1\le p \le \infty$, we could introduce a pseudometric $d_{S}^{(p)}$ on the barcodes of finite bounded type as follows.

\begin{definition}\label{ds}
	Let $\mathcal{B}$ and $\mathcal{C}$ be two barcodes of finite bounded type, and let $\Lambda_{\mathcal{B}} = \{\lambda_{i}\}_{i=1}^{\infty}$ and $\Lambda_{\mathcal{C}}  = \{\mu_{j}\}_{j=1}^{\infty}$ be the sequences of eigenvalues induced by the barcodes $\mathcal{B}$ and $\mathcal{C}$, respectively. Define the pseudo-distance between $\mathcal{B}$ and $\mathcal{C}$ by
	\[
	d_{S}^{(p)}(\mathcal{B}, \mathcal{C}) := \inf\limits_{\sigma}\sqrt[p]{\sum\limits_{i=1}^{\infty} \vert\lambda_{i}-\mu_{\sigma(i)}\vert^{p}},
	\]
	where $\sigma$ is taken all over the matchings between the two eigenvalue sequences $\Lambda_{\mathcal{B}}$ and $\Lambda_{\mathcal{C}}$.
\end{definition}

\begin{remark}
	Notice that if the barcode $\mathcal{B}$ is not of finite bounded type, the Gram matrix $G_{\mathcal{B}}$ may be unbounded or may have continuous spectrum. Then, the above process is only valid for barcodes of finite bounded type. However, in practical applications and calculations, we are actually dealing with the cases of finite bounded type induced by persistent homology groups. So, we haven't lost anything important.
\end{remark}

For any $1\le p \le \infty$, we will show that $d_{S}^{(p)}$ is a pseudometric on the set of all barcodes of finite bounded type. Recall that a pseudometric is a metric except for that the distance between two different elements is possible to be zero.

\begin{theorem}
	For any $1\le p \le \infty$, $d_{S}^{(p)}$ is a pseudometric on the set of all barcodes of finite bounded type.
\end{theorem}

\begin{proof}
	For any barcodes $\mathcal{B}$ and $\mathcal{C}$ of finite bounded type, it follows from the Definition \ref{ds} that $d_{S}^{(p)}(\mathcal{B}, \mathcal{C})\ge 0$, $d_{S}^{(p)}(\mathcal{B}, \mathcal{B})=0$ and $d_{S}^{(p)}(\mathcal{B}, \mathcal{C})=d_{S}^{(p)}(\mathcal{C}, \mathcal{B})$. Now, it suffices to prove the triangle inequality for $d_{S}^{(p)}$.
	
	Given any three barcodes $\mathcal{B}$, $\mathcal{C}$ and $\mathcal{D}$. Let $\widetilde{G}_{\mathcal{B}}$, $\widetilde{G}_{\mathcal{C}}$ and $\widetilde{G}_{\mathcal{D}}$  be the sequences of eigenvalues induced by $\mathcal{B}$, $\mathcal{C}$ and $\mathcal{D}$, respectively. Suppose that $\Lambda_{\mathcal{B}} = \{\lambda_{i}\}_{i=1}^{\infty}$, $\Lambda_{\mathcal{C}} = \{\mu_{j}\}_{j=1}^{\infty}$ and $\Lambda_{\mathcal{D}} = \{\nu_{k}\}_{k=1}^{\infty}$ are the eigenvalue sequences of $\widetilde{G}_{\mathcal{B}}$, $\widetilde{G}_{\mathcal{C}}$ and $\widetilde{G}_{\mathcal{D}}$, respectively.
	
	Since each sequence of eigenvalues only has finite non-zero elements, there exist two matchings $\sigma_{1}$ and  $\sigma_{2}$ such that
	\[
	d_{S}^{(p)}(\mathcal{B},\mathcal{C}) = \inf\limits_{\sigma}\sqrt[p]{\sum\limits_{i=1}^{\infty} \vert\lambda_{i}-\mu_{\sigma(i)}\vert^{p}}
	=\sqrt[p]{\sum\limits_{i=1}^{\infty} \vert\lambda_{i}-\mu_{\sigma_{1}(i)}\vert^{p}}
	\]
	and
	\begin{align*}
		d_{S}^{(p)}(\mathcal{C}, \mathcal{D}) &= \inf\limits_{\sigma}\sqrt[p]{\sum\limits_{i=1}^{\infty} \vert\mu_{i}-\nu_{\sigma(i)}\vert^{p}}\\
		&=\sqrt[p]{\sum\limits_{i=1}^{\infty} \vert\mu_{i}-\nu_{\sigma_{2}\sigma_{1}^{-1}(i)}\vert^{p}} \\
		&=\sqrt[p]{\sum\limits_{i=1}^{\infty} \vert\mu_{\sigma_{1}(i)}-\nu_{\sigma_{2}(i)}\vert^{p}}.
	\end{align*}
	Then,
	\begin{align*}
		d_{S}^{(p)}(\mathcal{B},\mathcal{D}) &= \inf\limits_{\sigma}\sqrt[p]{\sum\limits_{i=1}^{\infty} \vert\lambda_{i}-\nu_{\sigma(i)}\vert^{p}}\\
		&\leq \sqrt[p]{\sum\limits_{i=1}^{\infty} \vert\lambda_{i}-\nu_{\sigma_{2}(i)}\vert^{p}}\\
		&\leq \sqrt[p]{\sum\limits_{i=1}^{\infty} \vert\lambda_{i}-\mu_{\sigma_{1}(i)}\vert^{p}}+ \sqrt[p]{\sum\limits_{i=1}^{\infty} \vert\mu_{\sigma_{1}(i)}-\nu_{\sigma_{2}(i)}\vert^{p}}\\
		&= d_{S}^{(p)}(\mathcal{B}, \mathcal{C})+d_{S}^{(p)}(\mathcal{C}, \mathcal{D}).
	\end{align*}
	This finishes the proof.
\end{proof}

It is an important case of  $p=2$. Considering a matrix $\widetilde{G}_{\mathcal{B}}$, we could arrange its eigenvalues $\Lambda_{\mathcal{B}} = \{\lambda_{i}\}_{i=1}^{\infty}$ in descending order, i.e., $\lambda_{1}\ge \lambda_{2}\ge\cdots$. Then, for any two sequence of eigenvalues $\Lambda_{\mathcal{B}} = \{\lambda_{i}\}_{i=1}^{\infty}$ and $\Lambda_{\mathcal{C}} = \{\mu_{j}\}_{j=1}^{\infty}$ in descending order, the distance $d_{S}^{(2)}(\mathcal{B}, \mathcal{C})$ is achieved in the matching $id:\mathbb{N}\rightarrow\mathbb{N}$. This provides a simple formula for the pseudometric $d_{S}^{(2)}$.

\begin{proposition}\label{l2case}
	Let $\mathcal{B}$ and $\mathcal{C}$ be two barcodes of finite bounded type, and let $\Lambda_{\mathcal{B}} = \{\lambda_{i}\}_{i=1}^{\infty}$ and $\Lambda_{\mathcal{C}} = \{\mu_{j}\}_{j=1}^{\infty}$ be two sequences of eigenvalues  satisfying $\lambda_{1}\geq \lambda_{2}\geq \ldots$ and $\mu_{1}\geq \mu_{2}\geq \ldots$. Then,
	\[
	d_{S}^{(2)}(\mathcal{B}, \mathcal{C})=\inf\limits_{\sigma}\sqrt{ \sum\limits_{i=1}^{\infty} \vert\lambda_{i}-\mu_{\sigma(i)}\vert^{2}} = \sqrt{ \sum\limits_{i=1}^{\infty} \vert\lambda_{i}-\mu_{i}\vert^{2}}.
	\]
\end{proposition}
\begin{proof}
	Since $\mathcal{B}$ and $\mathcal{C}$ are barcodes of finite bounded type, $\Lambda_{\mathcal{B}} = \{\lambda_{i}\}_{i=1}^{\infty}$ and $\Lambda_{\mathcal{C}} = \{\mu_{j}\}_{j=1}^{\infty}$ are two sequences of non-negative numbers with finite positive elements. Consequently,  there exists  a positive integer $N$ such that
	\[
	\inf\limits_{\sigma} \sqrt{\sum\limits_{i=1}^{\infty} \vert\lambda_{i}-\mu_{\sigma(i)}\vert^{2}} = \min\limits_{\pi}  \sqrt{\sum\limits_{i=1}^{N} \vert\lambda_{i}-\mu_{\pi(i)}\vert^{2}},
	\]
	where $\sigma$ is taken all over the matchings between the two eigenvalue sequences $\Lambda_{\mathcal{B}}$ and $\Lambda_{\mathcal{C}}$, and $\pi$ is taken all over the permutations on $\{1,2, \cdots, N\}$. Then, it suffices to prove
	
	\begin{equation}\label{l2f}
		\min\limits_{\pi}  \sqrt{\sum\limits_{i=1}^{N} \vert\lambda_{i}-\mu_{\pi(i)}\vert^{2}}= \sqrt{ \sum\limits_{i=1}^{N} \vert\lambda_{i}-\mu_{i}\vert^{2}}.
	\end{equation}
	Let us prove the equation (\ref{l2f}) by the inductive method of $N$. If 
	$N=1$, it is obvious that the equation (\ref{l2f}) holds. 
	Now, assume that the equation (\ref{l2f}) holds for all $N=1,2, \cdots,n-1$, i.e., for any $k=1,2, \cdots,n-1$,
	\begin{equation}\label{assume}
		\min\limits_{\pi_k}  \sqrt{\sum\limits_{i=1}^{k} \vert\lambda_{i}-\mu_{\pi_k(i)}\vert^{2}}=\sqrt{\sum\limits_{i=1}^{k} \vert\lambda_{i}-\mu_{i}\vert^{2}},
	\end{equation}
	where $\pi_k$ is taken all over the permutations on $\{1,2, \cdots, k\}$.
	
	Consider $N = n$. Given any permutation $s$ on $\{1,2, \cdots, N\}$. We will show that  
	\[
	\sqrt{\sum\limits_{i=1}^{n} \vert\lambda_{i}-\mu_{s(i)}\vert^{2}}\ge \sqrt{\sum\limits_{i=1}^{n} \vert\lambda_{i}-\mu_{i}\vert^{2}}.
	\]
	
	If $s(1)=n$, then it follows the assumption (\ref{assume}) and the monotone decreasing of $\lambda_{i}$ and $\mu_{i}$ that
	\begin{align*}
		&\sum\limits_{i=1}^{n} \vert\lambda_{i}-\mu_{s(i)}\vert^{2}- \sum\limits_{i=1}^{n} \vert\lambda_{i}-\mu_{i}\vert^{2} \\
		=&(\lambda_{1}-\mu_{n})^{2}+\sum\limits_{i=2}^{n} (\lambda_{i}-\mu_{s(i)})^{2}- \sum\limits_{i=1}^{n} (\lambda_{i}-\mu_{i})^{2} \\
		\ge & \left((\lambda_{1}-\mu_{n})^{2}+\sum\limits_{i=1}^{n-1} (\lambda_{i+1}-\mu_{i})^{2}\right)-\left((\lambda_{n}-\mu_{n})^{2}+\sum\limits_{i=1}^{n-1} (\lambda_{i}-\mu_{i})^{2}\right)\\
		=& \left((\lambda_{1}-\mu_{n})^{2}-(\lambda_{n}-\mu_{n})^{2}\right)+\sum\limits_{i=1}^{n-1} \left((\lambda_{i+1}-\mu_{i})^{2}-(\lambda_{i}-\mu_{i})^{2}\right)\\
		=& \left(\lambda_{1}^{2}-\lambda_{n}^{2}-2\mu_{n}(\lambda_{1}-\lambda_{n})\right)+
		\sum\limits_{i=1}^{n-1}\left(\lambda_{i+1}^{2}-\lambda_{i}^{2}+2\mu_{i}(\lambda_{i}-\lambda_{i+1})\right)\\
		=&-2\mu_{n}(\lambda_{1}-\lambda_{n})+\sum\limits_{i=1}^{n-1}2\mu_{i}(\lambda_{i}-\lambda_{i+1})\\
		\ge&-2\mu_{n}(\lambda_{1}-\lambda_{n})+\sum\limits_{i=1}^{n-1}2\mu_{n}(\lambda_{i}-\lambda_{i+1})\\
		=& 0.
	\end{align*}
	If $s(1)=p<n$, then it follows the assumption (\ref{assume}) that
	\begin{align*}
		&\sum\limits_{i=1}^{n} \vert\lambda_{i}-\mu_{s(i)}\vert^{2} \\
		=& (\lambda_{1}-\mu_{p})^{2}+\sum\limits_{i=2}^{n} (\lambda_{i}-\mu_{s(i)})^{2} \\
		\ge & (\lambda_{1}-\mu_{p})^{2}+\sum\limits_{i=1}^{p-1} (\lambda_{i+1}-\mu_{i})^{2}+\sum\limits_{i=p+1}^{n} (\lambda_{i}-\mu_{i})^{2} \\
		\ge& \sum\limits_{i=1}^{p} (\lambda_{i}-\mu_{i})^{2}+\sum\limits_{i=p+1}^{n} (\lambda_{i}-\mu_{i})^{2} \\
		=&\sum\limits_{i=1}^{n} (\lambda_{i}-\mu_{i})^{2}.
	\end{align*}
	
	This finishes the proof.
\end{proof}

Moreover, the pseudometric $d_{S}^{(2)}$ is invariant in the sense of similarity. More precisely, if $\widetilde{G}_{\mathcal{B}}$ is similar to $\widetilde{G}_{\mathcal{B'}}$ and $\widetilde{G}_{\mathcal{C}}$ is similar to $\widetilde{G}_{\mathcal{C'}}$, then $d_{S}^{(2)}(\mathcal{B}, \mathcal{C})=d_{S}^{(2)}(\mathcal{B'}, \mathcal{C'})$. We show the similarity invariance of  the pseudometric $d_{S}^{(2)}$ as the following theorem.

\begin{theorem}\label{similar_orbit}
	Let $\mathcal{B}$ and $\mathcal{C}$ be two barcodes of finite bounded type, and let $\widetilde{G}_{\mathcal{B}}$ and $\widetilde{G}_{\mathcal{C}}$ be the matrices induced by $\mathcal{B}$ and $\mathcal{C}$, respectively. Let $\mathscr{S}(\widetilde{G}_{\mathcal{B}})$ and $\mathscr{S}(\widetilde{G}_{\mathcal{C}})$ be the similarity orbits of $\widetilde{G}_{\mathcal{B}}$ and $\widetilde{G}_{\mathcal{C}}$. Then
	$$
	d_{F}(\mathscr{S}(\widetilde{G}_{\mathcal{B}}),\mathscr{S}(\widetilde{G}_{\mathcal{C}}))=d_{S}^{(2)}(\widetilde{G}_{\mathcal{B}}, \widetilde{G}_{\mathcal{C}}).
	$$
\end{theorem}

\begin{proof}
	Let $\Lambda_{\mathcal{B}} = \{\lambda_{i}\}_{i=1}^{\infty}$ and $\Lambda_{\mathcal{C}} = \{\mu_{j}\}_{j=1}^{\infty}$ be the eigenvalue sequences of $\widetilde{G}_{\mathcal{B}}$ and $\widetilde{G}_{\mathcal{C}}$ satisfying $\lambda_{1}\geq \lambda_{2}\geq \ldots$ and $\mu_{1}\geq \mu_{2}\geq \ldots$. Since $\mathcal{B}$ and $\mathcal{C}$ are barcodes of finite bounded type, we could assume that ${G}_{\mathcal{B}}/\|{G}_{\mathcal{B}}\|$ is a $m\times m$  matrix and ${G}_{\mathcal{C}}/\|{G}_{\mathcal{C}}\|$ is a $n\times n$  matrix. Then, for any $k\ge N=\max\{m, n\}$, $\lambda_{k}=\mu_{k}=0$. By the Proposition \ref{l2case},
	\[
	d_{S}^{(2)}(\mathcal{B}, \mathcal{C})=\inf\limits_{\sigma}\sqrt{ \sum\limits_{i=1}^{\infty} \vert\lambda_{i}-\mu_{\sigma(i)}\vert^{2}} = \sqrt{ \sum\limits_{i=1}^{\infty} \vert\lambda_{i}-\mu_{i}\vert^{2}}= \sqrt{ \sum\limits_{i=1}^{N} \vert\lambda_{i}-\mu_{i}\vert^{2}}.
	\]
	Consequently, for any $T\sim \widetilde{G}_{\mathcal{B}}$ and $S\sim \widetilde{G}_{\mathcal{C}}$, it follows from Hoffman-Wielandt Theorem (Theorem \ref{Hoff} in the present paper) that there exits a permutation $\pi:\{1,2,\cdots, N\}\rightarrow\{1,2,\cdots, N\}$ such that
	\[
	\inf\limits_{\sigma}\sqrt{ \sum\limits_{i=1}^{\infty} \vert\lambda_{i}-\mu_{\sigma(i)}\vert^{2}} =  \sqrt{\sum\limits_{i=1}^{N}\vert\mu_{\pi(i)}-\lambda_{i}\vert^{2}}\leq \Vert S-T \Vert_{F}.
	\]
	Therefore, 
	\[
	d_{S}^{(2)}(\mathcal{B}, \mathcal{C})\le d_{F}(\mathscr{S}(\widetilde{G}_{\mathcal{B}}),\mathscr{S}(\widetilde{G}_{\mathcal{C}})).
	\]
	
	On the other hand, let
	\[
	D_{\mathcal{B}}=\begin{pmatrix}
		\lambda_{1} & 0 & \cdots & 0& 0&\cdots\\
		0 & \lambda_{2} & \cdots & 0 & 0&\cdots\\
		\vdots & \vdots & \ddots & \vdots & \vdots\\
		0 & 0 & \cdots & \lambda_{n} &0&\cdots\\
		0 & 0 & \cdots & 0 &0 & \cdots\\
		\vdots & \vdots &\cdots&\vdots&\vdots&\ddots
	\end{pmatrix}
	\ \ \ \text{and} \ \ \
	D_{\mathcal{C}}=\begin{pmatrix}
		\mu_{1} & 0 & \cdots & 0& 0 &\cdots\\
		0 & \mu_{2} & \cdots & 0 & 0 &\cdots\\
		\vdots & \vdots & \ddots & \vdots & \vdots\\
		0 & 0 & \cdots & \mu_{m} & 0 &\cdots\\
		0 & 0 & \cdots & 0 &0 & \cdots\\
		\vdots & \vdots &\cdots&\vdots&\vdots&\ddots
	\end{pmatrix}.
	\]
	Then, 
	\[
	\widetilde{G}_{\mathcal{B}}\sim D_{\mathcal{B}} \ \ \ \text{and} \ \ \ \widetilde{G}_{\mathcal{C}}\sim D_{\mathcal{C}}.
	\]
	Therefore,
	\begin{align*}
		d_{F}(\mathscr{S}(\widetilde{G}_{\mathcal{B}}),\mathscr{S}(\widetilde{G}_{\mathcal{C}}))&\leq  \Vert D_{\mathcal{B}}-D_{\mathcal{C}}\Vert_{F}\\
		&=\sqrt{ \sum\limits_{i=1}^{N} \vert\lambda_{i}-\mu_{i}\vert^{2}}\\
		&=\sqrt{ \sum\limits_{i=1}^{\infty} \vert\lambda_{i}-\mu_{i}\vert^{2}} \\
		&=d_{S}^{(2)}(\mathcal{B}, \mathcal{C}).
	\end{align*}
	The proof is completed.
\end{proof}

For more information about similarity orbits, readers can refer to \cite{Herrero-1982, Apo-Fia-Her-Voi-1984}.

In the application, the pseudometric $d_{S}^{(1)}$ is often useful too. Similar to the Proposition \ref{l2case}, if we arrange every sequences of eigenvalues in descending order, the distance $d_{S}^{(1)}(\mathcal{B}, \mathcal{C})$ is also achieved in the matching $id:\mathbb{N}\rightarrow\mathbb{N}$.

\begin{proposition}\label{l1case}
	Let $\mathcal{B}$ and $\mathcal{C}$ be two barcodes of finite bounded type, and let $\Lambda_{\mathcal{B}} = \{\lambda_{i}\}_{i=1}^{\infty}$ and $\Lambda_{\mathcal{C}} = \{\mu_{j}\}_{j=1}^{\infty}$ be two sequences of eigenvalues  satisfying $\lambda_{1}\geq \lambda_{2}\geq \ldots$ and $\mu_{1}\geq \mu_{2}\geq \ldots$. Then,
	\[
	d_{S}^{(1)}(\mathcal{B}, \mathcal{C})=\inf\limits_{\sigma} ~\sum\limits_{i=1}^{\infty} \vert\lambda_{i}-\mu_{\sigma(i)}\vert =  \sum\limits_{i=1}^{\infty} \vert\lambda_{i}-\mu_{i}\vert.
	\]
\end{proposition}

\begin{proof}
	Since $\mathcal{B}$ and $\mathcal{C}$ are barcodes of finite bounded type, $\Lambda_{\mathcal{B}} = \{\lambda_{i}\}_{i=1}^{\infty}$ and $\Lambda_{\mathcal{C}} = \{\mu_{j}\}_{j=1}^{\infty}$ are two sequences of non-negative numbers with finite positive elements. Consequently,  there exists  a positive integer $N$ such that
	\[
	\inf\limits_{\sigma} ~\sum\limits_{i=1}^{\infty} \vert\lambda_{i}-\mu_{\sigma(i)}\vert = \min\limits_{\pi} ~\sum\limits_{i=1}^{N} \vert\lambda_{i}-\mu_{\pi(i)}\vert,
	\]
	where $\sigma$ is taken all over the matchings between the two eigenvalue sequences $\Lambda_{\mathcal{B}}$ and $\Lambda_{\mathcal{C}}$, and $\pi$ is taken all over the permutations on $\{1,2, \cdots, N\}$. Then, it suffices to prove
	\begin{equation}\label{l1f}
		\min\limits_{\pi}~\sum\limits_{i=1}^{N} \vert\lambda_{i}-\mu_{\pi(i)}\vert= \sum\limits_{i=1}^{N} \vert\lambda_{i}-\mu_{i}\vert.
	\end{equation}
	Let us prove the equation (\ref{l1f}) by the inductive method of $N$. If
	$N=1$, it is obvious that the equation (\ref{l1f}) holds.
	Now, assume that the equation (\ref{l1f}) holds for all $N=1,2, \cdots,n-1$, i.e., for any $k=1,2, \cdots,n-1$,
	\begin{equation}\label{assume1}
		\min\limits_{\pi_k}~\sum\limits_{i=1}^{k} \vert\lambda_{i}-\mu_{\pi_k(i)}\vert=\sum\limits_{i=1}^{k} \vert\lambda_{i}-\mu_{i}\vert,
	\end{equation}
	where $\pi_k$ is taken all over the permutations on $\{1,2, \cdots, k\}$.
	
	Consider $N = n$. Given any permutation $s$ on $\{1,2, \cdots, N\}$. We will show that
	\[
	\sum\limits_{i=1}^{n} \vert\lambda_{i}-\mu_{s(i)}\vert \ge \sum\limits_{i=1}^{n} \vert\lambda_{i}-\mu_{i}\vert.
	\]
	
	Without loss of generality, suppose that $\lambda_n\ge \mu_n$.
	
	If $s(1)=n$, then it follows the assumption (\ref{assume1}) and the monotone decreasing of $\lambda_{i}$ and $\mu_{i}$ that
	\begin{align*}
		&\sum\limits_{i=1}^{n} \vert\lambda_{i}-\mu_{s(i)}\vert- \sum\limits_{i=1}^{n} \vert\lambda_{i}-\mu_{i}\vert\\
		=&\vert\lambda_{1}-\mu_{n}\vert+\sum\limits_{i=2}^{n} \vert\lambda_{i}-\mu_{s(i)}\vert- \sum\limits_{i=1}^{n} \vert\lambda_{i}-\mu_{i}\vert \\
		\ge & \left(\vert\lambda_{1}-\mu_{n}\vert+\sum\limits_{i=1}^{n-1} \vert\lambda_{i+1}-\mu_{i}\vert\right)-
		\left(\vert\lambda_{n}-\mu_{n}\vert+\sum\limits_{i=1}^{n-1} \vert\lambda_{i}-\mu_{i}\vert\right)\\
		=& \left(\vert\lambda_{1}-\mu_{n}\vert-\vert\lambda_{n}-\mu_{n}\vert\right)+\sum\limits_{i=1}^{n-1} \left(\vert\lambda_{i+1}-\mu_{i}\vert-\vert\lambda_{i}-\mu_{i}\vert\right)\\
		\ge& \vert\lambda_{1}-\lambda_{n}\vert-\sum\limits_{i=1}^{n-1}\vert\lambda_{i}-\lambda_{i+1}\vert\\
		=& (\lambda_{1}-\lambda_{n})-\sum\limits_{i=1}^{n-1}(\lambda_{i}-\lambda_{i+1})\\
		=& 0.
	\end{align*}
	If $s(1)=p<n$, then it follows the assumption (\ref{assume1}) that
	\begin{align*}
		&\sum\limits_{i=1}^{n} \vert\lambda_{i}-\mu_{s(i)}\vert \\
		=& \vert\lambda_{1}-\mu_{p}\vert+\sum\limits_{i=2}^{n} \vert\lambda_{i}-\mu_{s(i)}\vert \\
		\ge & \vert\lambda_{1}-\mu_{p}\vert+\sum\limits_{i=1}^{p-1} \vert\lambda_{i+1}-\mu_{i}\vert+\sum\limits_{i=p+1}^{n} \vert\lambda_{i}-\mu_{i}\vert \\
		\ge& \sum\limits_{i=1}^{p} \vert\lambda_{i}-\mu_{i}\vert+\sum\limits_{i=p+1}^{n} \vert\lambda_{i}-\mu_{i}\vert \\
		=&\sum\limits_{i=1}^{n} \vert\lambda_{i}-\mu_{i}\vert.
	\end{align*}
	
	This finishes the proof.
\end{proof}

At the end of this section, we give two examples to show that the pseudometric $d_{S}^{(2)}$ metric the topological spaces in the sense of "similarity". The first example shows that the pseudometric $d_{S}^{(2)}$ is concerned with the relative position relationship of the bars but not the congruent shape of the barcodes.

\begin{example}
	Let $\mathcal{B}=\{I_1, I_2\}$ and $\mathcal{C}=\{J_1, J_2\}$ be two barcodes, where $I_1=(1,3]$, $I_2=(2,4]$, $J_1=(6,8]$ and $J_2=(7,9]$. Then 
	\[
	G_{\mathcal{B}}=G_{\mathcal{C}}=\begin{pmatrix}
		4  & 1 \\
		1   & 4
	\end{pmatrix},
	\]
	and consequently
	\[
	\widetilde{G}_{\mathcal{B}}=\widetilde{G}_{\mathcal{C}}=\begin{pmatrix}
		\frac{4}{5} & \frac{1}{5} &0 &\cdots \\
		\frac{1}{5}   & \frac{4}{5} &0 &\cdots \\
		0  &0 &0 &\vdots \\
		\vdots &\vdots  &\vdots &\ddots
	\end{pmatrix}\sim 
	\begin{pmatrix}
		1 & 0 &0 &\cdots \\
		0   & \frac{3}{5} &0 &\cdots \\
		0  &0 &0 &\vdots \\
		\vdots &\vdots  &\vdots &\ddots
	\end{pmatrix}.
	\]
\end{example}

The second example shows that the pseudometric $d_{S}^{(2)}$ eliminates the influence of periodicity and measures the shape in a single period.

\begin{example}
	Let $\mathcal{B}=\{I_1, I_2\}$ and $\mathcal{C}=\{J_1, J_2, J_3, J_4\}$ be two barcodes, where $I_1=(1,3]$, $I_2=(2,4]$, $J_1=J_2=(1,3]$, $J_3=J_4=(2,4]$. Then
	\[
	G_{\mathcal{B}}=\begin{pmatrix}
		4  & 1 \\
		1   & 4
	\end{pmatrix} \ \ \ \text{and} \ \ \ \ 
	G_{\mathcal{C}}=\begin{pmatrix}
		4 &4  &1 & 1 \\
		4 &4  &1 & 1 \\
		1 &1  &4  & 4 \\
		1 &1  &4  & 4 
	\end{pmatrix},
	\]
	and consequently
	\[
	\widetilde{G}_{\mathcal{B}}=\begin{pmatrix}
		\frac{4}{5} & \frac{1}{5} &0 &\cdots \\
		\frac{1}{5}   & \frac{4}{5} &0 &\cdots \\
		0  &0 &0 &\vdots \\
		\vdots &\vdots  &\vdots &\ddots
	\end{pmatrix}
	\sim
	\begin{pmatrix}
		1 & 0 &0 &\cdots \\
		0   & \frac{3}{5} &0 &\cdots \\
		0  &0 &0 &\vdots \\
		\vdots &\vdots  &\vdots &\ddots
	\end{pmatrix} \sim 
	\begin{pmatrix}
		\frac{4}{5} &\frac{4}{5} & \frac{1}{5} &\frac{1}{5} &0 &\cdots  \\
		\frac{4}{5} &\frac{4}{5} &\frac{1}{5}   & \frac{1}{5} &0 &\cdots \\
		\frac{1}{5}   &\frac{1}{5}  & \frac{4}{5} & \frac{4}{5} &0 &\cdots \\	 
		\frac{1}{5}   &\frac{1}{5}  & \frac{4}{5} & \frac{4}{5} &0 &\cdots \\
		0  &0 &0 &0 &0 &\vdots \\
		\vdots &\vdots &\vdots &\vdots &\vdots &\ddots
	\end{pmatrix}=\widetilde{G}_{\mathcal{C}}.
	\]
	
\end{example}

\section{Computational examples}

%In the real world, one object which belong to a class should still belong to the same class after similarity transformations since it's corresponding normalized matrices that are similarity equivalent. 

As well known, conformal linear transformations are one representative type of similarity transformations, we provide specific examples to demonstrate that our pseudometrics $d_{S}^{(2)}$ and $d_{S}^{(1)}$ are stable under conformal linear transformations.  

Firstly, we apply $d_{S}^{(1)} $ and  $d_{S}^{(2)}$ to the barcodes generated by Vietoris-Rips filtration on synthetic dataset, and compare the performances of our distance $d_{S}$, bottleneck distance $d_{b}$ and Wasserstein distance $d_{\mathcal{W}}$.

Secondly, we apply $d_{S}^{(1)} $ and  $d_{S}^{(2)}$ to the barcodes generated by sublevel set filtration on the waves of a tuning fork and a piano, and compare the performances of our distance $d_{S}^{(1)}$,  $d_{S}^{(2)} $, bottleneck distance $d_{b}$ and Wasserstein distance $d_{\mathcal{W}}$.

Our code for the experiments is publicly available at \url{https://github.com/HeJiaxing-hjx/Stable_Similarity_Comparison_of_Persistent_Homology_Groups}.

\subsection{Comparison of $d_{b}$, $d_{\mathcal{W}}^{(1)}$, $d_{\mathcal{W}}^{(2)}$, $d_{S}^{(1)}$ and $d_{S}^{(2)}$ on Vietoris-Rips filtration using Clustering Classification}

In this subsection, we construct synthetic data set consisting of point clouds which are randomly sampled from two ellipses and three ellipses with and without conformal linear transformations. And we compare the performances of  our pseudometric $d_{S}^{(2)}$ $(d_{S}^{(1)})$, the bottleneck distance $d_{b}$ and Wasserstein distance $d_{\mathcal{W}}$ for a binary classification of the synthetic data set by using hierarchical clustering methods. 

\subsubsection{Construction of synthetic dataset}

Our synthetic data set consists of two classes: two ellipses and three ellipses. Each of our samples contains 240 points. We produce 150 point clouds of 240 points randomly sampled from each class. For the class with three ellipses, we have one ellipse centered at $(0,80)$ with a major axis of 35 and a minor axis of 25, from which we randomly sampled 40 points. Additionally, there are two ellipses centered at $(0,0)$, one equips with a major axis of 50 and a minor axis of 30, the other equips with a major axis of 40 and a minor axis of 20, and then we randomly sample 100 points from each ellipse. For the class with two ellipses, we have two ellipses centered at $(0,0)$, one with a major axis of 50 and a minor axis of 30, and the other with a major axis of 40 and a minor axis of 20, from which we randomly sampled 120 points each. Then, we add a level of Gaussian noise $(\eta = 0.5)$. We apply a conformal linear transformation to each sample, specifically, the rotation angles around the origin are uniformly sampled from the interval $[0,2\pi)$, the scaling factors are uniformed sampled from the interval $[0.1,10)$, and the two coordinates of translation are uniformly sampled from the interval $[0,1)$. This gives 300 point clouds without transformation and 300 point clouds with transformations in total. The Figure \ref{fig:vr_without_transform} below shows examples of samples from the two classes without transformation, and the Figure \ref{fig:vr2} shows examples of samples from the two classes with random conformal linear transformations.

%\begin{figure}[h]
%	\centering
%	\includegraphics[width=0.5\textwidth]{Two_classes_without_transformation.png}
%	\caption{Two classes examples without transformation. The left one shows  an example of random sampling on three ellipses, the right one shows an example of random sampling on two ellipses.}
%	\label{fig:vr_without_transform}
%\end{figure}
\begin{figure}[htbp]
	\begin{minipage}[t]{0.45\linewidth}
		\centering
		\includegraphics[height=5.5cm,width=5.5cm]{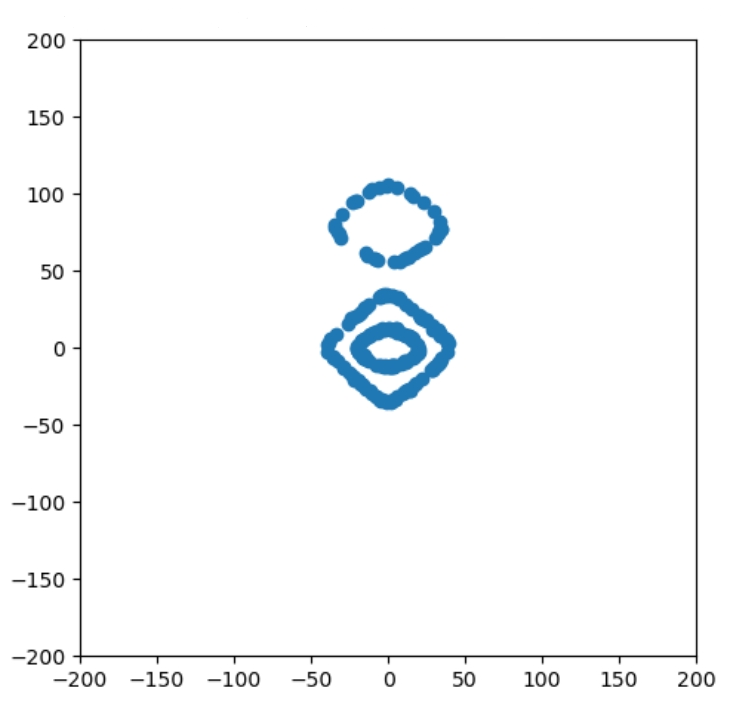}
		%\caption{$\mathscr{B}_{H(X)^{0}}$}
		%\label{fig:img_left}
	\end{minipage}%
	\begin{minipage}[t]{0.45\linewidth}
		\centering
		\includegraphics[height=5.5cm,width=5.5cm]{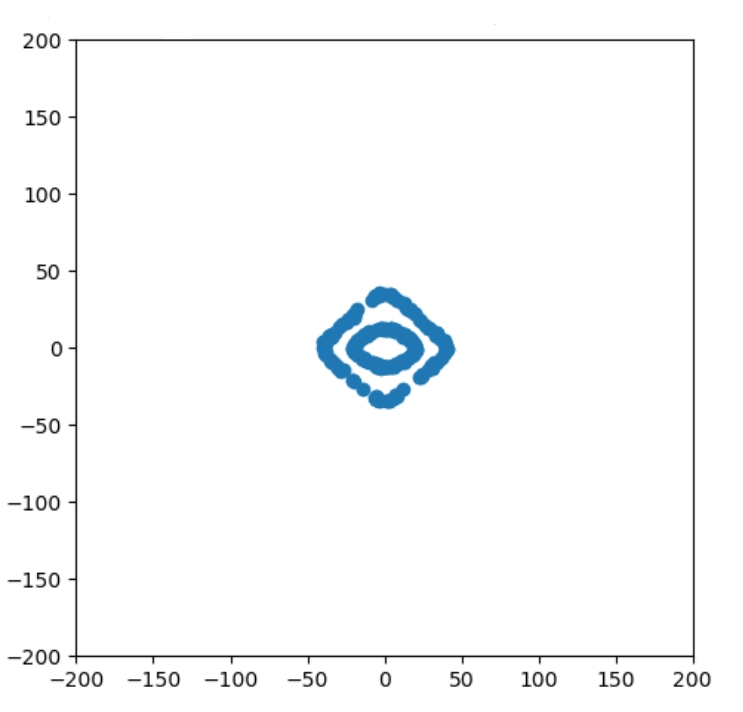}
		%\caption{$\mathscr{B}_{H(X)^{1}}$}
		%\label{fig:img_right}
	\end{minipage}
	\caption{Examples of two classes without transformation. The left one is an example of random sampling on three ellipses, the right one is an example of random sampling on two ellipses.}
	\label{fig:vr_without_transform}
\end{figure}

%\begin{figure}[h]
%	\centering
%	\includegraphics[width=0.5\textwidth]{Two_classes_with_transformation.png}
%	\caption{Two classes examples with conformal linear transformations. The left one shows the left example of Figure \ref{fig:vr_without_transform} with random conformal linear transformation, the right one shows the right example of Figure \ref{fig:vr_without_transform} with random conformal linear transformation.}
%	\label{fig:vr2}
%\end{figure}

\begin{figure}[htbp]
	\begin{minipage}[t]{0.45\linewidth}
		\centering
		\includegraphics[height=5.5cm,width=5.5cm]{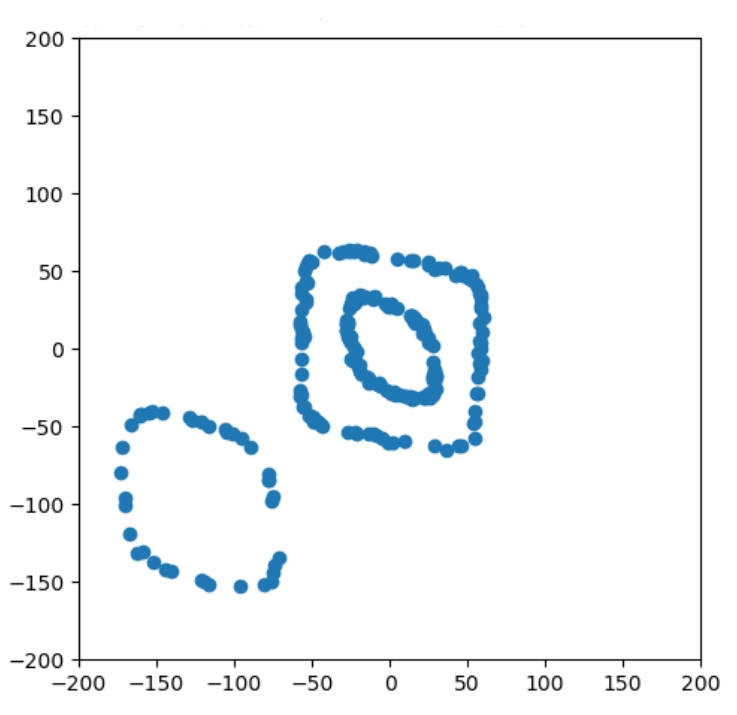}
		%\caption{$\mathscr{B}_{H(X)^{0}}$}
		%\label{fig:img_left}
	\end{minipage}%
	\begin{minipage}[t]{0.45\linewidth}
		\centering
		\includegraphics[height=5.5cm,width=5.5cm]{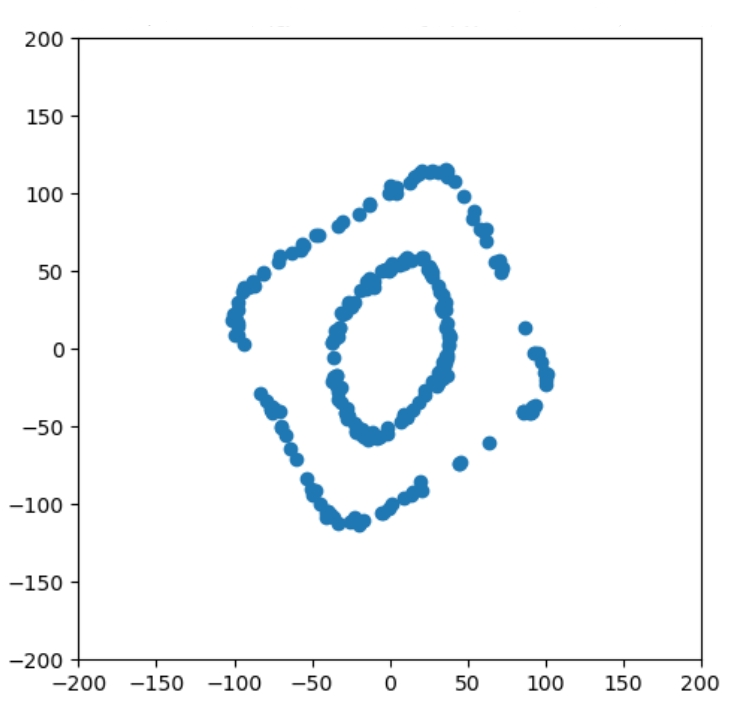}
		%\caption{$\mathscr{B}_{H(X)^{1}}$}
		%\label{fig:img_right}
	\end{minipage}
	\caption{Examples of two classes with conformal linear transformations. The left one is the left example of Figure \ref{fig:vr_without_transform} with random conformal linear transformation, the right one is the right example of Figure \ref{fig:vr_without_transform} with random conformal linear transformation.}
	\label{fig:vr2}
\end{figure}

\subsubsection{Construction of comparable experiments}

We build Vietoris-Rips filtration from each point cloud which have been endowed with the ambient Euclidean metric on $\mathbb{R}^{2}$, and use the one dimensional homology functor $H_{1}$ to produce persistence barcodes which detect loops in our filtered topological space. We  use $d_{b}$, $d_{\mathcal{W}}^{(1)}$ and  $d_{\mathcal{W}}^{(2)}$ to compare with our pseudometrics $d_{S}^{(1)}$ and $d_{S}^{(2)}$ with one homological dimension ($H_{1}$).

There are three clustering methods which require only distance matrices for clustering including $K$-medoids Clustering \cite{Rdu-Kau-1987}, Agglomerative Clustering \cite{Ward-1963} and Affinity Propagation Clustering \cite{affinity-2007}. $K$-medoids produces a patition of a metric space into $K$ clusters by choosing $K$ points from the data set called medoids and assigning each metric space point to its closest medoids.  The Agglomerative Clustering  performs a hierarchical clustering using a bottom up approach (each observation starts in its own cluster, and clusters are successively merged together). Affinity Propagationis a clustering algorithm which can find "exemplars", members of the input set that are representative of clusters. If one pseudometric is fixed, we obtain the corresponding distance matrix in which each element is the pairwise distance between two samples. We compare these distance matrices based on how well they classify the random point clouds into two classes via the three clustering methods.

In our experiments, we set $K=2$ in $K$-medoids Clustering.  In Agglomerative Clustering, we set the number of clusters equal to 2 and use average-linkage. The linkage criterion determines how the distance between sets of observations is calculated. Specifically, with average-linkage, the distance between two clusters is defined as the average of all pairwise distances between points in the two clusters. In our experiments, we set the affinity of Affinity Propagation to be precomputed.  

\subsubsection{Evaluation metric}

We evaluate the clustering results using the Fowlkes-Mallows Index ($FMI$) \cite{FMI-1983} and the numbers of the clusters. The Fowlkes-Mallows Index  is defined by the geometric mean between the precisions,
$$FMI = \frac{TP}{\sqrt{((TP + FP) * (TP + FN))}},   $$
where $TP$ is the number of true positives (i.e. the number of pair of points that belongs to the same clusters in both true labels and predicted labels), $FP$ is the number of false positives (i.e. the number of pair of points that belongs in the same clusters in true labels and not in predicted labels) and $FN$ is the number of false negatives (i.e. the number of pair of points that belongs in the same clusters in predicted labels and not in true labels). The score ranges from 0 to 1. A high value indicates a good similarity between two clusters.

\subsubsection{Results}
In each experiment, we perform 100 trials, average the classification scores ($FMI$) and average the number of clusters. In Table \ref{table1}, we use $FMI$ to report the classification accuracy of the three clustering methods with and without conformal linear transformations. We can see that the $FMI$ scores for our pseudometrics $d_{S}^{(1)}$ and $d_{S}^{(2)}$ are almost identical with and without conformal transformations, which means our pseudometrics remain stable under conformal linear transformations. However, the scores for $d_{b}$, $d_{\mathcal{W}}^{(1)}$ and $d_{\mathcal{W}}^{(2)}$ significantly decrease when applying conformal linear transformations to the dataset. Therefore, on dataset subjected to conformal linear transformations, our method clearly outperforms the other three pseudometrics. Additionally, although $d_{b}$, $d_{\mathcal{W}}^{(1)}$ and $d_{\mathcal{W}}^{(2)}$ perform better than ours in agglomerative  clustering without transformations, their scores are still very low in Affinity Propagation Clustering. Thus, in the Affinity Propagation Clustering method, our pseudometrics clearly outperform the other three pseudometrics.

\begin{table*}[h]\small
	%\begin{table*}[h]
	\centering
	
	%\begin{adjustbox}{center}
	%\begin{tabular}{|c|p{1.5cm}|p{1.5cm}|p{1.5cm}|p{1.5cm}|}
	%\begin{tabular}{cp{1cm}p{1cm}p{1cm}p{1cm}}
	\begin{tabular}{|c|c|c|c|c|c|c|}
		
		\hline
		\multicolumn{1}{|c|}{\multirow{2}{*}{Pseudometric}}
		&\multicolumn{3}{|c|}{Without transformation }&\multicolumn{3}{|c|}{With transformation}
		\\
		
		%		\multicolumn{1}{|c|}{}
		%		&\multicolumn{2}{|c|}{(Accuracy)}&\multicolumn{2}{|c|}{(Accuracy)}
		%		\\
		\cline{2-7} 
		
		\multicolumn{1}{|c|}{} &{KMC}& { AC} & {APC}&{KMC}&{AC } & {APC}
		\\
		
		\hline
		
		\multicolumn{1}{|c|}{$d_{b}$} &{0.943}& {0.946} & {0.526} &{0.528} & {0.652} & {0.629}\\
		\hline
		
		\multicolumn{1}{|c|}{$d_{\mathcal{W}}^{(1)}$} &\textbf{0.958} & {0.948} & {0.577} &{0.518} & {0.626} & {0.520}\\
		\hline
		
		\multicolumn{1}{|c|}{$d_{\mathcal{W}}^{(2)}$} &{0.946} & \textbf{0.949} & {0.669} &{0.521} & {0.652} & {0.596}\\
		\hline
		
		\multicolumn{1}{|c|}{$d_{S}^{(1)}$} &{0.859} & {0.836} & {0.729} &\textbf{0.852} & \textbf{0.844} & {0.721}\\
		\hline
		
		\multicolumn{1}{|c|}{$d_{S}^{(2)}$} &{0.843} & {0.812} & \textbf{0.738} &\textbf{0.852} & {0.822} & \textbf{0.743}\\
		\hline

	\end{tabular}
	%\end{adjustbox}
	%\label{table_MAP}
	\vspace{2mm}
	\caption{Average $FMI$ scores of $K$-medoids Clustering (KMC), Agglomerative Clustering (AC) and Affinity Propagation Clustering (APC) for calculating pseudometrics between each pair of samples on synthetic dataset. Bold indicates highest scores. }
	\label{table1}
	
\end{table*}

In Table \ref{table2}, we record the number of clusters obtained from Affinity Propagation Clustering. We can see that only our metrics yield  numbers of clusters equal to 2, which indicate that our pseudometrics are superior to the other three pseudometrics $d_{b}$, $d_{\mathcal{W}}^{(1)}$ and $d_{\mathcal{W}}^{(2)}$. 

\begin{table*}[h]\small
	%\begin{table*}[h]
	\centering
	
	%\begin{adjustbox}{center}
	%\begin{tabular}{|c|p{1.5cm}|p{1.5cm}|p{1.5cm}|p{1.5cm}|}
	%\begin{tabular}{cp{1cm}p{1cm}p{1cm}p{1cm}}
	\begin{tabular}{|c|c|c|}
		
		\hline
		\multicolumn{1}{|c|}{\multirow{2}{*}{Pseudometric}}
		&\multicolumn{1}{|c|}{Without transformation }&\multicolumn{1}{|c|}{With transformation}
		\\
		
		\multicolumn{1}{|c|}{}
		&\multicolumn{1}{|c|}{(APC)}&\multicolumn{1}{|c|}{(APC)}
		\\
		%\cline{2-3} 
		\hline

		\multicolumn{1}{|c|}{$d_{b}$} & {64.66}  & {3.39} \\
		\hline
		
		\multicolumn{1}{|c|}{$d_{\mathcal{W}}^{(1)}$} & {6.27} & {4.07} \\
		\hline
		
		\multicolumn{1}{|c|}{$d_{\mathcal{W}}^{(2)}$} & {4.07} & {3.48} \\
		\hline
		
		\multicolumn{1}{|c|}{$d_{S}^{(1)}$} & {2.07} & {2.07} \\
		\hline
		
		\multicolumn{1}{|c|}{$d_{S}^{(2)}$} & \textbf{2.00} & \textbf{2.00}\\
		\hline

	\end{tabular}
	%\end{adjustbox}
	%\label{table_MAP}
	\vspace{2mm}
	\caption{Number of clusters generated by Affinity Propagation Clustering (APC) algorithms on synthetic dataset. Bold represents the correct number of classes.}
	\label{table2}
	
\end{table*}

In Table 3, we record the time required for calculating the distance between each pair of samples. All timings are computed on a laptop with an AMD Ryzen 7 5800H with Radeon Graphics and 16GB of memory. We compute $d_{b}$ distance matrices using Python packages from
\url{https://persim.scikit-tda.org/}. And we  compute $d_{\mathcal{W}}^{(1)}$ and $d_{\mathcal{W}}^{(2)}$ distance matrices using the software named HERA \cite{Kerber-2016} . The average running time for the distance matrices under  $d_{S}^{(1)}$ and $d_{S}^{(2)}$ is significantly less than that for the distance matrices under $d_{b}$ and is comparable to the average running time for the distance matrices under $d_{\mathcal{W}}^{(1)}$ and $d_{\mathcal{W}}^{(2)}$ with HERA acceleration.

\begin{table*}[h]\small
	%\begin{table*}[h]
	\centering
	
	%\begin{adjustbox}{center}
	%\begin{tabular}{|c|p{1.5cm}|p{1.5cm}|p{1.5cm}|p{1.5cm}|}
	%\begin{tabular}{cp{1cm}p{1cm}p{1cm}p{1cm}}
	\begin{tabular}{|c|c|c|}
		
		\hline
		\multicolumn{1}{|c|}{\multirow{2}{*}{Pseudometric}}
		&\multicolumn{1}{|c|}{Without transformation }&\multicolumn{1}{|c|}{With transformation}
		\\
		
		\multicolumn{1}{|c|}{}
		&\multicolumn{1}{|c|}{(Time)}&\multicolumn{1}{|c|}{(Time)}
		\\
		%\cline{2-3} 
		\hline

		\multicolumn{1}{|c|}{$d_{b}$} & {566s}  & {515s} \\
		\hline
		
		\multicolumn{1}{|c|}{$d_{\mathcal{W}}^{(1)}$} & {49s} & {34s} \\
		\hline
		
		\multicolumn{1}{|c|}{$d_{\mathcal{W}}^{(2)}$} & {93s} & {56s} \\
		\hline
		
		\multicolumn{1}{|c|}{$d_{S}^{(1)}$} & {63s} & {55s} \\
		\hline
		
		\multicolumn{1}{|c|}{$d_{S}^{(2)}$} & {64s} & {56s}\\
		\hline

	\end{tabular}
	%\end{adjustbox}
	%\label{table_MAP}
	\vspace{2mm}
	\caption{Average running time for calculating pseudometrics between each pair of samples on synthetic dataset.}
	\label{table3}
	
\end{table*}

\subsection{Comparison of $d_{b}$, $d_{\mathcal{W}}^{(1)}$, $d_{\mathcal{W}}^{(2)}$, $d_{S}^{(1)}$ and $d_{S}^{(2)}$ on sublevel set filtration using Clustering Classification}

To verify that our metrics can distinguish different waveforms of sound waves,  we conduct experiments on the wave of a 440 Hz tuning fork with resonance box and the wave of a single note from a piano obtained from the Freesound online sound archive \cite{tuning, piano}. The experimental results demonstrate that our pseudometric is only related to the essential waveform and is independent on frequency and amplitude.

%In order to validate the candidate annotations, we used Freesound Datasets \cite{freesound}, an online platform for the collaborative creation of open audio datasets developed at the Music Technology Group.

\subsubsection{Construction of dataset}

%By randomly selecting 100 waveforms of duration 0.02 seconds from the tuning fork dataset within the time range $[3, 15]$, and randomly selecting 100 waveforms of duration 0.02 seconds from the piano dataset within the time range $[0.2, 1]$, we slice the tuning fork and piano data to get 200 waveforms in total. The Figure \ref{Tuning} below shows two waveforms  examples of the tuning fork and the Figure \ref{Piano} shows two waveforms  examples of the piano. The reason for selecting waveforms with a duration of 0.02 seconds is that sound waves are periodic, and each 0.02-second waveform slice contains one complete cycle.

We slice the waves of the tuning fork and piano to get 200 segments  with durations in $[0.02, 0.05)$ of waves in total, which are obtained by randomly selecting 100 segments of the wave from the tuning fork dataset in the time interval $[3, 15)$ and selecting 100 segments of the wave from the piano dataset in the time interval $[0.2, 1)$.  There are two reasons for selecting sound waves with a duration between [0.02, 0.05). First, sound waves are periodic, and a slice longer than 0.02 seconds must contain a complete cycle. Second, we want to verify that our pseudometric is independent on the frequency or the number of periods, so we choose different length of the slices. We take different beginning time of the slices whose amplitudes are different. The Figure \ref{Tuning} below shows two examples of the waves of the tuning fork and the Figure \ref{Piano} shows two   examples of the waves of the piano.

\begin{figure}[htbp]
	\centering
	\begin{minipage}{0.95\linewidth}
		\centering
		\includegraphics[width=0.9\linewidth]{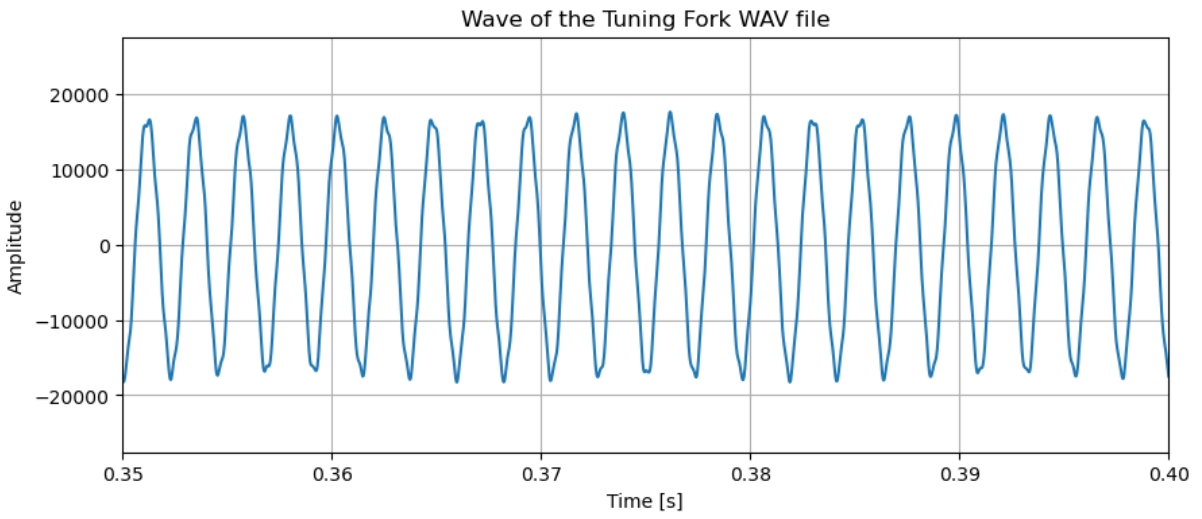}
		%\caption{}
		%\label{chutian1}
	\end{minipage}
	\qquad
	
	\begin{minipage}{0.95\linewidth}
		\centering
		\includegraphics[width=0.9\linewidth]{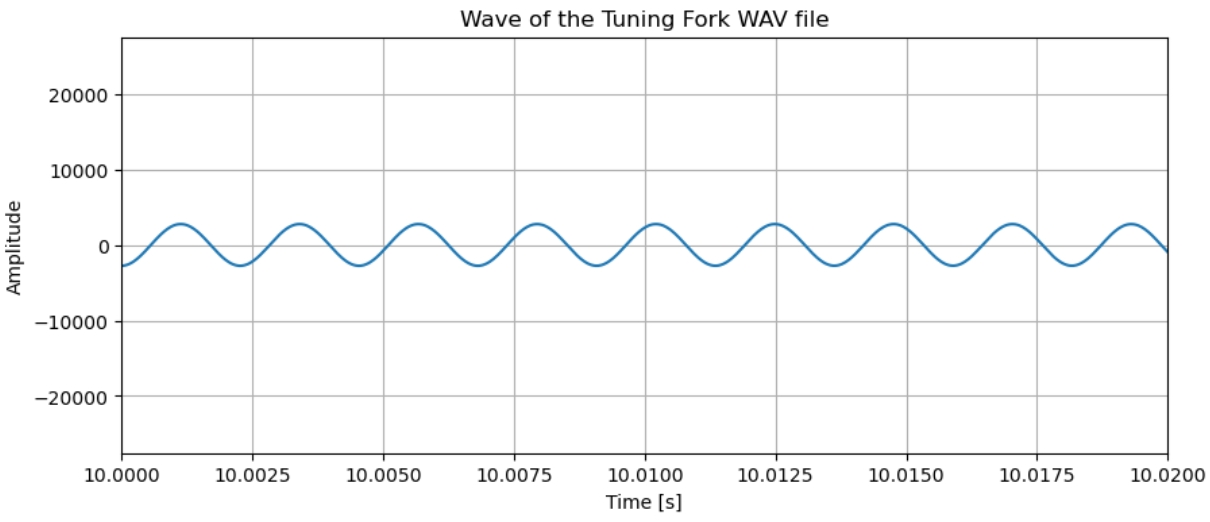}
		%\caption{chutian2}
		%\label{chutian2}
	\end{minipage}
	\caption{Waves produced by a tuning fork with different amplitudes.}
	\label{Tuning}
\end{figure}

\begin{figure}[htbp]
	\centering
	\begin{minipage}{0.95\linewidth}
		\centering
		\includegraphics[width=0.9\linewidth]{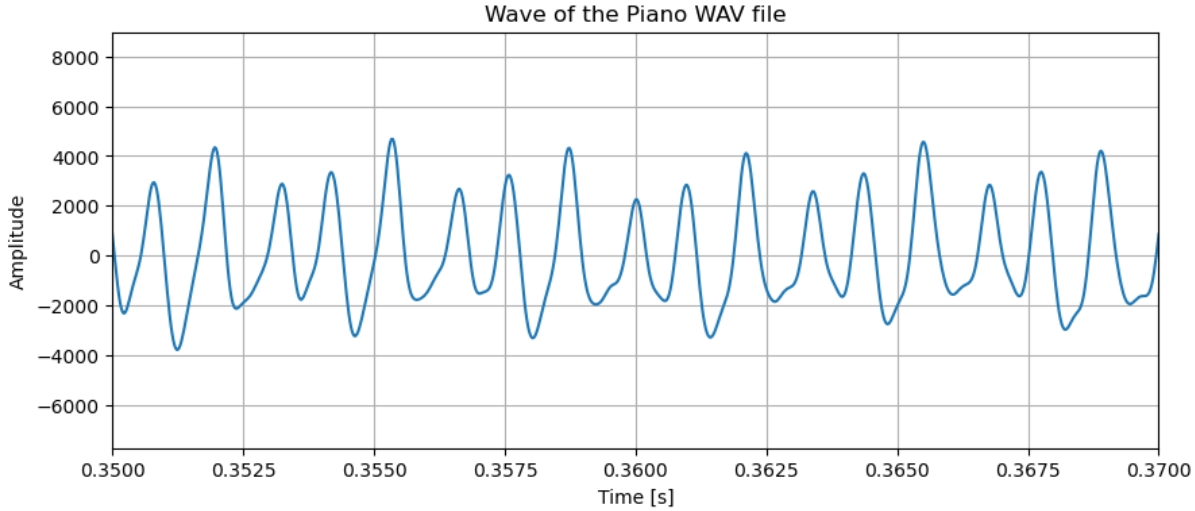}
		%\caption{}
		%\label{chutian1}
	\end{minipage}
	\qquad
	
	\begin{minipage}{0.95\linewidth}
		\centering
		\includegraphics[width=0.9\linewidth]{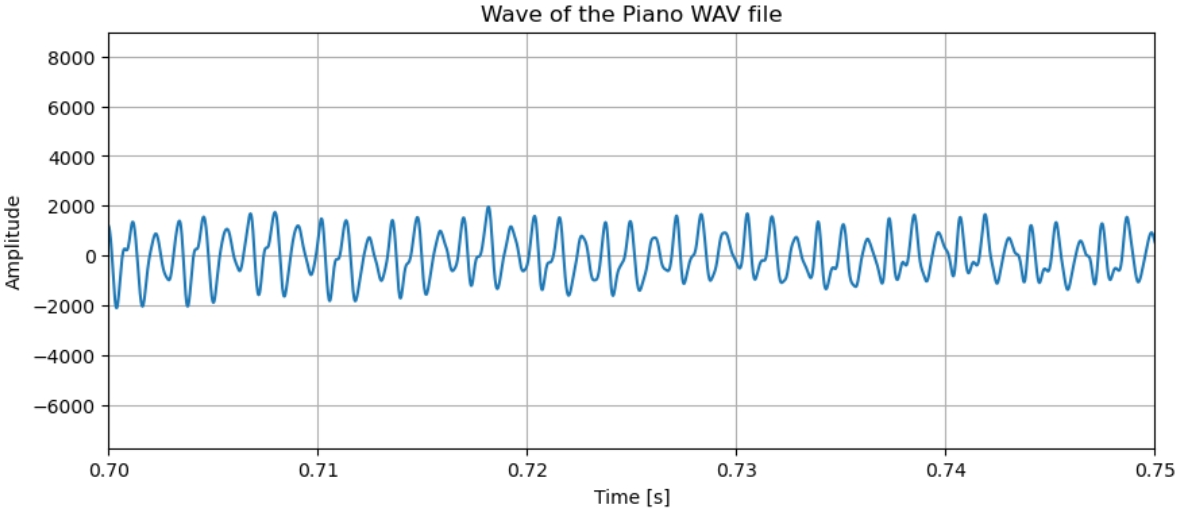}
		%\caption{chutian2}
		%\label{chutian2}
	\end{minipage}
	\caption{Waves produced by a piano with different amplitudes.}
	\label{Piano}
\end{figure}

\subsubsection{Construction of comparable experiments}

We build sublevel set filtration by Ripser software \cite{Ripser} from each wave and apply the zero dimensional homology functor $H_{0}$ to produce barcodes which detect connect components.

In the present experiment, we use the same clustering methods and evaluation metric as in the synthetic dataset experiment.

\subsection{Results}

In each experiment, we perform 100 trials, average the classification scores ($FMI$) and average the number of clusters. In Table \ref{table4}, we report the classification accuracies of the three clustering methods for the waves of the tuning fork and piano. Our pseudometric $d_{S}^{(1)}$ achieve an $FMI$ score of 1 in Agglomerative Clustering and $K$-medoids Clustering, while $d_{S}^{(2)}$ can achieve an $FMI$ score of 0.997 in $K$-medoids Clustering. Our pseudometrics $d_{S}^{(1)}$ and $d_{S}^{(2)}$ also achieve higher $FMI$ scores in Affinity Propagation than the other three pseudometrics $d_{b}$, $d_{\mathcal{W}}^{(1)}$ and $d_{\mathcal{W}}^{(2)}$. Our pseudometrics $d_{S}^{(1)}$ and $d_{S}^{(2)}$ outperform the other three pseudometrics $d_{b}$, $d_{\mathcal{W}}^{(1)}$ and $d_{\mathcal{W}}^{(2)}$, since it can distinguish the sound waves produced by different instruments, even the amplitudes and frequencies of the sound waves change over time. The experimental results demonstrate that our pseudometric is only related to the essential waveform and is independent on frequency and amplitude.

\begin{table*}[h]\small
	%\begin{table*}[h]
	\centering
	
	%\begin{adjustbox}{center}
	%\begin{tabular}{|c|p{1.5cm}|p{1.5cm}|p{1.5cm}|p{1.5cm}|}
	%\begin{tabular}{cp{1cm}p{1cm}p{1cm}p{1cm}}
	\begin{tabular}{|c|c|c|c|}
		
		\hline
		\multicolumn{1}{|c|}{Pseudometric}&{KMC}& { AC} & {APC}\\
		
		%		\multicolumn{1}{|c|}{}
		%		&\multicolumn{2}{|c|}{(Accuracy)}&\multicolumn{2}{|c|}{(Accuracy)}
		%		\\
		%\cline{2-4} 

		\hline
		
		\multicolumn{1}{|c|}{$d_{b}$} &{0.539}& {0.699} & {0.603} \\
		\hline
		
		\multicolumn{1}{|c|}{$d_{\mathcal{W}}^{(1)}$} &{0.539} & {0.666} & {0.643}\\
		\hline
		
		\multicolumn{1}{|c|}{$d_{\mathcal{W}}^{(2)}$}&{0.534} & {0.669} & {0.650}\\
		\hline
		
		\multicolumn{1}{|c|}{$d_{S}^{(1)}$} &\textbf{1.000} & \textbf{1.000} & \textbf{0.874}\\
		\hline
		
		\multicolumn{1}{|c|}{$d_{S}^{(2)}$} &{0.997} & {0.889} & {0.669}\\
		\hline

	\end{tabular}
	%\end{adjustbox}
	%\label{table_MAP}
	\vspace{2mm}
	\caption{Average $FMI$ scores  of $K$-medoids Clustering (KMC), Agglomerative Clustering (AC) and Affinity Propagation Clustering (APC) for calculating  pseudometrics between each pair of samples from the waves generated by a tuning fork and a piano. Bold indicates highest scores. }
	\label{table4}
	
\end{table*}
In Table \ref{table5}, we record the number of clusters obtained from Affinity Propagation Clustering. We can see that only our pseudometrics yield  numbers of clusters equal to 2, which indicate our pseudometrics are superior to the other three pseudometrics $d_{b}$, $d_{\mathcal{W}}^{(1)}$ and $d_{\mathcal{W}}^{(2)}$. 

\begin{table*}[h]\small
	%\begin{table*}[h]
	\centering
	
	%\begin{adjustbox}{center}
	%\begin{tabular}{|c|p{1.5cm}|p{1.5cm}|p{1.5cm}|p{1.5cm}|}
	%\begin{tabular}{cp{1cm}p{1cm}p{1cm}p{1cm}}
	\begin{tabular}{|c|c|}
		
		\hline
		\multicolumn{1}{|c|}{Pseudometric} &{APC}
		\\

		%\cline{2-3} 
		\hline

		\multicolumn{1}{|c|}{$d_{b}$} & {2.29}  \\
		\hline
		
		\multicolumn{1}{|c|}{$d_{\mathcal{W}}^{(1)}$}  & {2.36}  \\
		\hline
		
		\multicolumn{1}{|c|}{$d_{\mathcal{W}}^{(2)}$}  & {2.09}  \\
		\hline
		
		\multicolumn{1}{|c|}{$d_{S}^{(1)}$}  & \textbf{2.00}  \\
		\hline
		
		\multicolumn{1}{|c|}{$d_{S}^{(2)}$}  & \textbf{2.00} \\
		\hline

	\end{tabular}
	%\end{adjustbox}
	%\label{table_MAP}
	\vspace{2mm}
	\caption{Number of clusters generated by Affinity Propagation Clustering (APC) algorithms on samples from the waves generated by a tuning fork and a piano. Bold represents the correct number of classes.}
	\label{table5}
	
\end{table*}

In Table \ref{table6}, we record the time required for calculating the distance between each pair of samples. The average running time for the distance matrices under $d_{S}^{(1)}$ and $d_{S}^{(2)}$ is significantly less than that for the distance matrices under the $d_{b}$ distance.

\begin{table*}[h]\small
	%\begin{table*}[h]
	\centering
	
	%\begin{adjustbox}{center}
	%\begin{tabular}{|c|p{1.5cm}|p{1.5cm}|p{1.5cm}|p{1.5cm}|}
	%\begin{tabular}{cp{1cm}p{1cm}p{1cm}p{1cm}}
	\begin{tabular}{|c|c|}
		
		\hline
		\multicolumn{1}{|c|}{Pseudometric} &{Time}\\

		%\cline{2-3} 
		\hline

		\multicolumn{1}{|c|}{$d_{b}$} & {891s}   \\
		\hline
		
		\multicolumn{1}{|c|}{$d_{\mathcal{W}}^{(1)}$}& {36s} \\
		\hline
		
		\multicolumn{1}{|c|}{$d_{\mathcal{W}}^{(2)}$} & {67s} \\
		\hline
		
		\multicolumn{1}{|c|}{$d_{S}^{(1)}$} & {82s} \\
		\hline
		
		\multicolumn{1}{|c|}{$d_{S}^{(2)}$} & {86s} \\
		\hline

	\end{tabular}
	%\end{adjustbox}
	%\label{table_MAP}
	\vspace{2mm}
	\caption{Average running time for calculating different pseudometrics between each pair of samples on from the waves generated by a tuning fork and a piano.}
	\label{table6}
	
\end{table*}

\section{Conclusion and future work}

In this paper, we define a pseudometric $d_{S}^{(p)}$ and provide $d_{S}^{(2)}$ is a similarity invariant. The pseudometric $d_{S}^{(2)}$  $(d_{S}^{(1)})$ not only can be used to classify the objects in the sense of similarity, but also can reduce the computation time  for pairwise distances. 

In the future work, we would like to further develop our pseudometric on both theoretical and practical perspectives. We will further study the pseudometric $d_{S}^{(p)}$ in Matrix Theory and  Operator Theory. As well as the method to classify the timbres in this paper previously, we are confident that our pseudometric $d_{S}^{(p)}$ has the potential to deal with issues related to waveforms, such as time series analysis, signal processing and so on. Furthermore, it can be extended to suit more tasks in graphics, imaging processing and artificial intelligence.

\section*{Declarations}

\noindent \textbf{Ethics approval}

\noindent Not applicable.

\noindent \textbf{Competing interests}

\noindent The author declares that there is no conflict of interest or competing interest.

\noindent \textbf{Authors' contributions}

\noindent All authors contributed equally to this work.

\noindent \textbf{Funding}

\noindent There is no funding source for this manuscript.

\bibliography{reference}
\end{document}